\documentclass[11pt]{amsart}

\usepackage{epsf,amssymb,times,overpic,amscd}
\usepackage[usenames]{color}
\usepackage[all]{xy}
\usepackage{hyperref}

\usepackage{amsfonts}
\usepackage{amsmath}
\usepackage{graphicx}
\usepackage{epsfig}
\usepackage{bbm}

\usepackage{color}
\newcommand{\ba}{\begin{array}}
\newcommand{\ea}{\end{array}}
\newcommand{\be}{\begin{enumerate}}
\newcommand{\ee}{\end{enumerate}}

\newtheorem{thm}{Theorem}[section]
\newtheorem{prop}[thm]{Proposition}
\newtheorem{lemma}[thm]{Lemma}
\newtheorem{cor}[thm]{Corollary}
\newtheorem{conj}[thm]{Conjecture}

\newtheorem{prob}[thm]{Problem}

\theoremstyle{definition}
\newtheorem{defn}[thm]{Definition}

\theoremstyle{remark}

\newtheorem{rmk}[thm]{Remark}

\newcommand{\mb}{\mathbf{1}}
\newcommand{\C}{\mathcal{C}}
\newcommand{\K}{\mathbf{k}}
\newcommand{\Hom}{\operatorname{Hom}}
\newcommand{\End}{\operatorname{End}}
\newcommand{\HOM}{\operatorname{HOM}}
\newcommand{\END}{\operatorname{END}}
\newcommand{\oplusop}[1]{{\mathop{\oplus}\limits_{#1}}}
\newcommand{\CC}{\mathbb{C}}
\newcommand{\BB}{\mathbb{B}}
\newcommand{\Bi}{\mathrm{PB}}  
\newcommand{\lra}{\longrightarrow}
\newcommand{\Kvect}{\K\mathrm{-vect}}
\newcommand{\Ztwo}{{\Z\big[\frac{1}{2}\big]}}
\newcommand{\Zen}{{\Z\big[\frac{1}{n}\big]}}

\newcommand{\ZNxi}{{\Z\big[\frac{1}{N},\xi\big]}}
\newcommand{\Zpxi}{{\Z\big[\frac{1}{p},\xi\big]}}
\newcommand{\umod}{{\textrm{--}\underline{\mathrm{mod}}}}

\newcommand{\R}{\mathbb{R}}
\newcommand{\Z}{\mathbb{Z}}
\newcommand{\Q}{\mathbb{Q}}
\newcommand{\N}{\mathbb{N}}

\newcommand{\bdry}{\partial}
\newcommand{\n}{\noindent}
\newcommand{\ot}{\otimes}
\newcommand{\we}{\wedge}
\newcommand{\ra}{\rightarrow}
\newcommand{\xra}{\xrightarrow}
\newcommand{\mf}{\mathbf}
\newcommand{\op}{\operatorname}
\newcommand{\Ker}{\operatorname{Ker}}
\newcommand{\IM}{\operatorname{Im}}
\newcommand{\cal}{\mathcal}
\newcommand{\es}{\emptyset}
\newcommand{\lan}{\langle}
\newcommand{\ran}{\rangle}

\newcommand{\pr}{\mathcal{P}}

\newcommand{\xn}{X^{\otimes n}}
\newcommand{\xm}{X^{\otimes m}}
\newcommand{\xk}{X^{\otimes k}}
\newcommand{\xnw}{{X^{\wedge}}^{\otimes n}}
\newcommand{\xkw}{{X^{\wedge}}^{\otimes k}}
\newcommand{\wtz}{z^*}

\newcommand{\lk}{L_k}
\newcommand{\rk}{R_{k-1}}

\newcommand{\Ka}{\operatorname{Ka}}
\newcommand{\B}{\mathcal{B}}
\newcommand{\wtc}{\mathcal{C}^{pre}}
\newcommand{\D}{\mathcal{D}}
\newcommand{\wt}{\widetilde}
\newcommand{\wtd}{\mathcal{D}^{pre}}
\newcommand{\E}{\mathcal{E}}
\newcommand{\T}{\mathcal{T}}

\begin{document}
\title{How to categorify the ring of integers localized at two}

\author{Mikhail Khovanov and Yin Tian}
\address{Department of Mathematics, Columbia University, New York, NY 10027}
\email{khovanov@math.columbia.edu}
\address{Yau Mathematical Sciences Center, Tsinghua University, Beijing 100084, China}
\email{ytian@math.tsinghua.edu.cn}

\maketitle

\begin{abstract}
We construct a triangulated monoidal Karoubi closed category with the Grothendieck
ring naturally isomorphic to the ring of integers localized at two.
\end{abstract}

\tableofcontents

\section{Introduction}

Natural numbers $\N$, integers $\Z$, rationals $\Q$, and real numbers $\R$ are basic
structures that belong to the foundations of modern mathematics.
The category $\Kvect$ of finite-dimensional vector spaces over a field $\K$ can be viewed
as a categorification of $\N=\{0,1,2, \dots \}$. The Grothendieck monoid of  $\Kvect$ is
naturally isomorphic to the semiring $\N$, via the map that sends the
image $[V]$ of a vector space $V$ in the monoid to its dimension, an element
of $\N$. Direct sum and tensor product of vector spaces
lift addition and multiplication in $\N$. Additive monoidal category $\Kvect$,
which linear algebra studies, is indispensable in modern mathematics and
its applications, even more so when the field $\K$ is $\R$ or $\CC$.

The ring of integers $\Z$ is categorified via the
category $\mathcal{D}(\Kvect)$ of complexes of vector spaces up to chain
homotopies, with finite-dimensional total cohomology groups.
The Grothendieck ring $K_0$ of the category $\mathcal{D}(\Kvect)$ is
isomorphic to $\Z$ via the map that takes a symbol $[V]\in K_0$ of
a complex $V$ in $\mathcal{D}(\Kvect)$ to its Euler characteristic
$\chi(V)$. The notion of cohomology and the use of category
$\mathcal{D}(\Kvect)$ is ubiquous in modern mathematics as well.

The category $\mathcal{D}(\Kvect)$ is triangulated monoidal, with the
Grothendieck ring $\Z$. Given the importance of the simplest categorifications
of $\N$ and $\Z$, which are behind the subjects of linear algebra and
cohomology, it is natural to ask if the two objects next in complexity
on the original list, rings $\Q$ and $\R$, can be categorified.
More precisely, are there monoidal triangulated categories with
Grothendieck rings isomorphic to $\Q$ and $\R$? A natural additional
requirement is for the categories to be idempotent complete (Karoubi closed).

This question was recently asked in~\cite[Problem 2.3]{Kh2} for $\Q$ and without the
idempotent completeness requirement, together with
a related and potentially simpler problem~\cite[Problem 2.4]{Kh2}
to categorify the ring
$\Zen$, the localization of $\Z$ given by inverting $n$.

In the present paper we present an idempotent complete categorification
of the ring $\Ztwo$. In Section~\ref{dg-section}
we construct a DG (differential graded) monoidal category $\C$.
The full subcategory $D^c(\C)$ of compact objects of the derived category $D(\C)$ is an idempotent complete triangulated monoidal category.
We then establish a ring isomorphism
\begin{equation}\label{ring-iso-eq}
K_0(D^c(\C)) \ \cong \ \Z\bigg[ \frac{1}{2}\bigg] .
\end{equation}
The category $\C$ is
generated by the unit object $\mb$ and an additional object $X$ and
has a diagrammatic flavor. Morphisms between tensor powers of $X$ are
described via suitable planar diagrams modulo local relations.
Diagrammatic approach to monoidal categories is common in quantum
algebra and categorification, and plays a significant role in the present
paper as well.

\vspace{0.1in}

Let us provide a simple motivation for the construction.
To have a monoidal category with the Grothendieck ring
$\Ztwo$ one would want an object $X$ whose image
$[X]$ in the Grothendieck ring is $\frac{1}{2}$. The simplest
way to try to get $[X]=\frac{1}{2}$ is via a direct sum
decomposition $\mb \cong X\oplus X $ in the category, since the symbol
$[\mb]$ of the unit object is the unit element $1$ of the Grothendieck
ring.

The relation $\mb\cong X \oplus X$ would imply that the endomorphism
algebra of $\mb$ is isomorphic to $\mathrm{M}_2(\End(X))$, the
$2\times 2$ matrix algebra with coefficients in the ring of endomorphisms
of $X$. This contradicts the property that the endomorphism ring of
the unit object $\mb$ in an additive monoidal category
is a commutative ring (or a super-commutative ring, for categories enriched
over super-vector spaces),
since matrix algebras are far from being commutative.

Our idea is to keep some form of the direct sum relation but rebalance to
have one $X$ on each side. Specifically, assume that the category
is triangulated and objects can be shifted by $[n]$. If there is a direct
sum decomposition
\begin{equation} \label{eq1}
  X \cong X[-1] \oplus \mb,
  \end{equation}
then in the Grothendieck ring there is a relation $[X] = -[X] + 1$,
equivalent to $2[X]=1$, that is, $X$ descends to the element $\frac{1}{2}$ in
the Grothendieck ring of the category.

Equation (\ref{eq1}) can be implemented by requiring mutually-inverse
isomorphisms between its two sides. Each of these isomorphisms has two components: a morphism between $X$ and its shift $X[-1]$, and a morphism
between $X$ and $\mb$, for the total of four morphisms, see Figure \ref{sec3fig5}.
Two morphisms between $X$ and $X[-1]$ of cohomological degrees zero give rise to two endomorphisms of $X$ of cohomological degrees $1$ and $-1$.
It results in a
DG (differential graded) algebra structure on the endomorphism ring of
$1\oplus X$, but with the trivial action of the differential. The four morphisms and relations on them
are represented graphically in Figures~\ref{sec3fig1} and \ref{sec3fig2} at the beginning of Section~\ref{dg-section}. These generators and relations give rise to a monoidal
DG category $\C$ with a single generator $X$ (in addition to
the unit object $\mb$). This category is studied in Section~\ref{dg-section},
where it is also explained how it gives rise to a triangulated monoidal idempotent complete category $D^c(\C)$.

Our main result is Theorem~\ref{thm main} in Section~\ref{Sec K}
stating
that the Grothendieck ring of $D^c(\C)$ is isomorphic to $\Ztwo$, when the ground field $\K$
has characteristic two.
To prove this result, we compute the Grothendieck group of the DG algebra
$A_k$ of endomorphisms of $X^{\otimes k}$ for all $k$.
Our computation requires determining the first K-groups of certain DG algebras and ends up being quite tricky.
Higher $K$-theory of DG algebras and categories
appears to be a subtle and difficult subject,
where even definitions need to be chosen carefully for basic computational
goals. This is the perception of the authors of the present paper, who only dipped their
toes into the subject.

In the proof we eventually need to specialize to
working over a field $\K$
of characteristic two, but the category is defined over
any field and the isomorphism (\ref{ring-iso-eq}) is likely to hold over any
field as well.

Section~\ref{basis-section} is devoted to preliminary work, used
in later sections, to construct a pre-additive monoidal category generated
by a single object $X$ (and the unit object $\mb$) subject
to additional restrictions that the endomorphism ring of $\mb$ is
the ground field $\K$ and the composition map
$$\Hom(X,\mb) \otimes_{\K} \Hom(\mb,X) \ \lra \ \Hom(X,X)$$
is injective. Such a structure can be encoded by the ring $A$ of
endomorphisms of $\mb\oplus X$ and the idempotent $e\in A$ corresponding
to the projection of $\mb\oplus X$ onto $\mb$ subject to
conditions that $eAe\cong \K$ and the multiplication
$(1-e)Ae \otimes eA(1-e) \lra (1-e)A(1-e)$ is injective.
We show that this data does generate a  monoidal category,
give a diagrammatical presentation for this category, and provide
a basis for the hom spaces $\Hom(X^{\otimes n}, X^{\otimes m})$.
From the diagrammatic viewpoint, categories of these type
are not particularly complicated, since the generating morphisms are
given by labels on single strands and labels on strands ending or
appearing inside the diagrams. No generating morphisms go between
tensor products of generating objects, which may give rise to complicated
networks built out of generating morphisms.

Our categorification of $\Ztwo$ in Section~\ref{dg-section} and
a conjectural categorification of $\Zen$ in Section~\ref{leavitt-section}
both rely on certain instances of $(A,e)$ data and on the monoidal
categories they generate. The construction of Section~\ref{dg-section} requires
us to work in the DG setting,  with the associated
triangulated categories and with the complexity of the
structure mostly happening on the homological side.

\vspace{0.05in}

In Section~\ref{leavitt-section} we propose another approach to
categorification of $\Ztwo$ and $\Zen$, based on an alternative way to
stabilize the impossible isomorphism  $\mb \cong X \oplus X$.
One would want an isomorphism
\begin{equation}\label{stabilize-iso-eq}
  X\cong  \mb \oplus X \oplus X \oplus X,
  \end{equation}
which does not immediately contradict $\End(\mb)$
being commutative or super-commutative.
We develop this approach in Section~\ref{leavitt-section}.
Object $X$ can be thought of as categorifying $-\frac{1}{2}$.
There are no
shift functors present in isomorphism (\ref{stabilize-iso-eq})
and it is possible to work
here with the usual $K_0$ groups of algebras. Interestingly, we
immediately encounter {\em Leavitt path algebras}, that have
gained wide prominence in ring theory, operator algebras and related fields
over the last decade, see~\cite{A} and references therein.

The Leavitt algebra $L(1,n)$ is a universal ring $R$ with the property that
$R\cong R^{n}$ as a left module over itself~\cite{L},
that is, the rank one free $R$-module is isomorphic to the rank $n$ free module.
Such an isomorphism is encoded by the entries of an $n\times 1$ matrix
$(x_1, \dots, x_n)^T$, giving a module map $R\lra R^n$, and the entries of
the $1\times n$ matrix $(y_1, \dots, y_n)$, giving a map $R^n\lra R$,
with $x_i,y_i$'s elements of $R$.
These maps being mutually-inverse isomorphisms produces a system of
equations on $x_i$'s and $y_i$'s, and the Leavitt algebra $L(1,n)$ is
the quotient of the free algebra on the $x_i$'s and $y_i$'s by these
relations. Leavitt algebras have exponential growth and are not noetherian.
They satisfy many remarkable properties and have found various
applications~\cite{A}. The relation to equation (\ref{stabilize-iso-eq})
is that, when ignoring the unit object (by setting any morphism
that factors through $\mb$ to zero), one would need an
isomorphism $X\cong X^3$, and the six endomorphisms of $X$ giving rise
to such an isomorphism satisfy the Leavitt algebra $L(1,3)$ defining relations.

Natural generalizations of the Leavitt algebras include Leavitt path
algebras~\cite{A} and Cohn-Leavitt algebras~\cite{AG,A} which can be encoded via
oriented graphs.
One can think of these algebras as categorifying certain systems of homogeneous linear equations with non-negative integer coefficients, see~\cite[Proposition 9]{A}, \cite{ArBrCo}.

For the category in Section~\ref{leavitt-section},
the algebra of endomorphisms of the direct sum $\mb \oplus X$
is a particular Leavitt path algebra $L(Q)$ associated to the
graph $Q$ given by
\begin{align} \label{quiver1}
\xymatrix{
X \ar@(ul,ur)[]^{1}  \ar@(dl,ul)[]^{2} \ar@(dr,dl)[]^{3} \ar[r] & Y
}
\end{align}
The Leavitt path algebra $L(Q)$ categorifies the
linear equation $x= 3x+y$. Quotient of this algebra by the two-sided
ideal generated by the idempotent of projecting $X\oplus Y$ onto $Y$
is isomorphic to the Leavitt algebra $L(1,3)$.

The construction of Section~\ref{leavitt-section} can be viewed
as forming a \emph{monoidal envelope}
of the Leavitt path algebra $L(Q)$, where $Y$ is set to be the unit object $\mb$.
Endomorphisms of $\mb\oplus X$ are encoded by $L(Q)$, and these endomorphism
spaces are then extended to describe morphisms between arbitrary
tensor powers of $X$.
Passing from certain Leavitt path algebras to monoidal categories can perhaps
be viewed as categorifications of quotients of free algebras by
certain systems of inhomogeneous linear equations with non-negative integer coefficients imposed on generators of free algebras.

We come short of proving that the Grothendieck ring of the associated
idempotent completion is indeed $\Ztwo$. The obstacle is in not knowing K-groups
$K_i(L(1,3)^{\otimes k})$ of tensor powers of $L(1,3)$ for $i=0,1$, see Conjecture \ref{conj lk}.
This problem is discussed in~\cite{ArBrCo}, but the answer is not known for general $k$.

Equation (\ref{stabilize-iso-eq}) admits a natural generalization to
\begin{equation}
  X\cong  \mb \oplus X^{n+1},
  \end{equation}
where the right hand side contains $n+1$ summands $X$.
Now $X$ plays the role of categorified $-\frac{1}{n}$.
In Section \ref{Sec Zen} we construct an additive monoidal Karoubi closed category in which the isomorphism above holds
and conjecture that its Grothendieck ring is isomorphic to $\Zen$, see Conjecture \ref{conj n}.

\vspace{0.1in}
The DG ring $A$ of endomorphisms of $\mb \oplus X$ that appears in
our categorification of $\Ztwo$ in Section~\ref{dg-section}
is also a Leavitt path algebra $L(T)$, of
the Toeplitz graph $T$ given by
\begin{align} \label{quiver2}
\xymatrix{
X \ar@(dl,ul)[] \ar[r] & Y
}
\end{align}
see~\cite[Example 7]{A}.
Considering $A$ as a Leavitt path algebra, one should
ignore the grading of $A$ and its structure as a DG algebra.
Leavitt path algebra of the Toeplitz graph $T$ is isomorphic to
the Jacobson algebra~\cite{J}, with generators
$a,b$ and defining relation $ba=1$, making $a$ and $b$ one-sided
inverses of each other.
The linear equation
categorified by this algebra is $x=x+y$, lifted
to an isomorphism of projective modules $X\cong X\oplus Y$. In the
Grothendieck group of this Leavitt path algebra $[Y]=0$.
In the monoidal envelope where $Y$ is the unit object $\mb$, an isomorphism $X\cong X\oplus \mb$ would imply the Grothendieck ring is zero, since $1=[\mb]=0$.
These problems are avoided by introducing a
shift into the isomorphism as in equation (\ref{eq1}), and consequently working in the DG framework, as explained earlier. Choice of $[-1]$ over $[1]$ is inessential, see Remark~\ref{remark-change-sign}.

\vspace{0.1in}

Having a monoidal structure or some close substitute is a natural requirement
for a categorification of $\Zen$ and $\Q$, emphasized in~\cite{Kh2}. The
direct limit $D=M_{n^{\infty}}(\K)$ of matrix algebras $M_{n^k}(\K)$ under
unital inclusions $M_{n^k}(\K)\subset M_{n^{k+1}}(\K)$ has the Grothendieck
group $K_0$ of finitely-generated projective modules naturally isomorphic
to the abelian group $\Zen$, see~\cite[Exercise 1.2.7]{Ro}. The isomorphism
is that of groups, not rings. Similar limits give algebras with $K_0$
isomorphic to any subgroup of $\Q$.

Phillips~\cite{Ph} shows that the algebra $D$ is algebraically strongly selfabsorbing,
that is, there is an isomorphism $D\cong D\otimes_{\K} D$ which is
algebraically approximately similar to the inclusion
$D\cong D\otimes 1\subset D\otimes_{\K}D$.
This isomorphism allows to equip $K_0(D)$ with a ring
structure, making $K_0(D)$ isomorphic to $\Zen$ as a ring, and
likewise for the other subrings of $\Q$, see~\cite{Ph}. We are not
aware of any
monoidal structure or its close substitute on the category of
finitely-generated projective $D$-modules that would induce
the Phillips ring structure on $K_0(D)$.

Barwick et al.~\cite{BGHNS} construct triangulated categories (and
stable $\infty$-categories) with Grothendieck groups isomorphic
to localizations $S^{-1}\Z$ of $\Z$ along any set $S$ of
primes, as well as more general localizations. For these localizations a
monoidal or some tensor product structure on the underlying categories
does not seem to be present, either,
to turn Grothendieck groups into rings.

\vspace{0.1in}

\n{\bf Acknowledgments.} 
The authors are grateful to the referee for a suggestion leading to Remark \ref{rmk char}.
Thank Bangming Deng, Peter Samuelson and Marco
Schlichting for helpful conversations.
M.K. was partially supported by the NSF grant DMS-1406065 while working on the paper. Y.T. was partially supported by
the NSFC 11601256.


\section{A family of monoidal categories via arc diagrams}
\label{basis-section}

\n {\bf A monoidal category from $(A,e)$.}
For a field $\K$, let $A$ be a unital $\K$-algebra and $e\in A$ an idempotent such
that $eAe\cong \K$ and the multiplication map
$$ (1-e)Ae \otimes eA(1-e) \stackrel{m'}{\longrightarrow} (1-e)A(1-e),$$
denoted $m'$, is injective. Another notation for $m'(b\otimes c)$ is
simply $bc$. Let
$$ A' \ = \mathrm{im}(m')\subset (1-e)A(1-e)$$
and choose a $\K$-vector subspace $A''$ of $A$ such that
 $(1-e)A(1-e)= A' \oplus A''$.

We fix bases $\BB_{1,0}$, $\BB_{0,1}$, and $\BB_{1,1}(1)$ of vector spaces
$(1-e)Ae$, $eA(1-e)$, and $A''$, respectively.
The subscript $0$ corresponds
to the idempotent $e$, and the subscript $1$ corresponds to the complementary idempotent
$1-e$.
The notation $\BB_{1,1}(1)$ will be explained later in (\ref{eq b11}).
We also choose $\BB_{0,0}$ to be a one-element set consisting of
any nonzero element of $\K\cong eAe$ (element $1$ is a natural choice).
The set $\BB_{1,0} \times \BB_{0,1}$ of elements $ bc$, over all $b\in \BB_{1,0}$ and $c\in \BB_{0,1},$ is naturally
a basis of $A'$. The union $\BB_{1,1}(1) \sqcup (\BB_{1,0} \times \BB_{0,1})$ gives a basis of $(1-e)A(1-e)$.
Choices of $A''$ and $\BB_{1,0}$, $\BB_{0,1}$, $\BB_{1,1}(1)$ are not needed in the
definition of category $\C$ below.

To a pair $(A,e)$ as above we will assign a $\K$-linear pre-additive strict monoidal category $\C=\C(A,e)$.
Objects of $\C$ are tensor powers $X^{\otimes n}$ of the generating object $X$. The unit
object $\mb = X^{\otimes 0}$. The algebra $A$ describes the ring of endomorphisms of the
object $\mb \oplus X$.
Slightly informally, we write $A$ in the matrix notation
\begin{equation}\label{eq-matrixA}
 A = \left( \begin{array}{cc} eAe & eA(1-e) \\
    (1-e)Ae & (1-e)A(1-e) \end{array} \right) .
\end{equation}
meaning, in particular, that, as a $\K$-vector space, $A$ is the
direct sum of the four matrix entries, and the  multiplication $A\otimes A \stackrel{m'}{\lra} A$  in $A$
reduces to matrix-like tensor product maps between the entries.
The two diagonal entries are subalgebras, via nonunital inclusions.
We declare (\ref{eq-matrixA}) to be the matrix of homs between the
summands of the object $\mb\oplus X$ in $\C$:
$$  \left( \begin{array}{cc} \Hom_{\C}(\mb,\mb) & \Hom_{\C}(\mb,X) \\
    \Hom_{\C}(X,\mb) & \Hom_{\C}(X,X) \end{array} \right)
    \ = \
    \left( \begin{array}{cc} eAe & eA(1-e) \\
   (1-e)Ae & (1-e)A(1-e) \end{array} \right) .$$

The algebra $eAe= \K$ is the endomorphism ring of the unit object $\mb$.
The $\K$-vector space $(1-e)Ae$ is $\Hom_{\C}(\mb, X)$, while
$\Hom_{\C}(X,\mb) = eA(1-e) $, and the ring $\End_{\C}(X) = (1-e)A(1-e).$

The algebra $A$ can be used to generate a vector space of morphisms
between tensor powers of $X$, by tensoring and composing morphisms
between the objects $\mb$ and $X$, and imposing only the relations
that come from the axioms of a strict monoidal pre-additive category.
The only nontrivial part, as explained below, is to check that the
category does not degenerate, that is, the
hom spaces in the resulting category have the expected sizes (bases).

\begin{figure}[h]
\includegraphics[height=2.3cm]{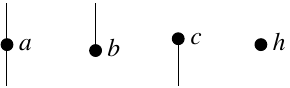}
\caption{Presentation of morphisms, boundary of $\R\times [0,1]$ not shown}
\label{generatorsfig}
\end{figure}

Before considering diagrammatics for morphisms between tensor powers of
$X$, we start with diagrams that describe homs between $\mb$ and $X$. We draw the diagrams inside a strip $\R\times [0,1]$.
An endomorphism $a$ of $X$ ($a$ is an element of $(1-e)A(1-e)$)
is depicted by a vertical line that starts and ends on the boundary of the strip and in the middle carries a dot labeled $a$, see Figure~\ref{generatorsfig}.
We call such a line a \emph{long strand} labeled by $a$.
An element $b\in \Hom_{\C}(\mb,X)=(1-e)Ae$ is depicted by a \emph{short top strand} labeled by $b$. The top endpoint of a short top strand is at the boundary.
An element $c\in \Hom_{\C}(X,\mb)=eA(1-e)$
is depicted by a \emph{short bottom strand} labeled by $c$. Its bottom endpoint is at the boundary of the strip. Each short strand has two endpoints: the boundary endpoint (either at the top or bottom of the strand), and the \emph{floating} endpoint, which is a labeled dot.
An endomorphism $h$ of
the identity object $\mb$ is depicted by a dot,
labeled by $h$, in the middle of the plane (in our case, these
endomorphisms are elements of the ground field $\K$). These four types of
diagrams are depicted in Figure~\ref{generatorsfig}.

  \begin{figure}[h]
    \includegraphics[height=5cm]{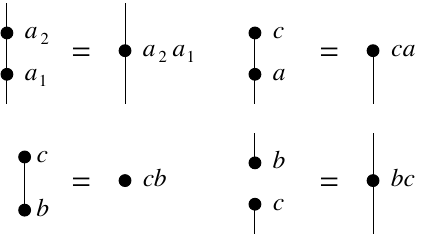}
    \caption{Graphical presentation of composition in A}
  \label{rel1fig}
   \end{figure}

Vertical concatenation of diagrams corresponds to the composition of
morphisms, as depicted in Figure~\ref{rel1fig}.
For instance, if an element of $(1-e)A(1-e)$ factors as $bc$, for
$b\in (1-e)Ae$ and $c\in eA(1-e)$, we can depict it as a composition
of a top strand with label $b$ and a bottom strand with label $c$, see the lower right equality above.

Addition of alike diagrams is given by adding their labels, see examples in
Figure~\ref{rel2fig} for adding elements of $(1-e)A(1-e)$ and $(1-e)Ae$.

\begin{figure}
  \includegraphics[height=1.5cm]{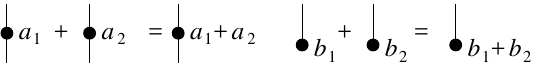}
  \caption{Adding diagrams}
  \label{rel2fig}
\end{figure}

Likewise, scaling a diagram by an element of $\K$ corresponds to
multiplying its label by that element.
An element $a\in (1-e)A(1-e)$ decomposes uniquely $a=a'+a''$,
where $a'\in A'=\mathrm{im}(m'), a''\in A''$. Furthermore, $a'$ admits a (non-unique)
presentation $a' \ = \ \sum_{i=1}^k  b_i c_i,$ $ b_i \in (1-e)Ae,$
$c_i \in eA(1-e),$ see Figure~\ref{rel3fig} for diagrammatic expression.

\begin{figure}
  \includegraphics[height=2.0cm]{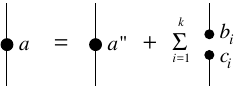}
  \caption{Decomposition of an element in $(1-e)A(1-e)$}
  \label{rel3fig}
\end{figure}

The algebra $A$ has a basis given by diagrams in Figure~\ref{generatorsfig}
over all $a\in \BB_{1,1}(1) \sqcup (\BB_{1,0} \times \BB_{0,1})$, $b\in \BB_{1,0},$ $c\in \BB_{0,1}$, and $h\in \BB_{0,0}$
(recall that $\BB_{0,0}$ has cardinality one).
Vertical line without a label denotes the idempotent $1-e$. This idempotent does not have to lie in $A''$, but we usually choose $A''$ to contain $1-e$ and a basis $\BB_{1,1}(1)$ of $A''$  to contain $1-e$ as well (also see Example 2 below).

These diagrammatics for $A$ extend to diagrammatics for a monoidal
category with the generating object $X$ and algebra $A$ describing the endomorphisms
of $\mb\oplus X$.
Morphisms from $X^{\otimes n}$ to $X^{\otimes m}$ are $\K$-linear
combinations of diagrams with $n$ bottom and $m$ top endpoints which
are concatenations of labeled long and short strands, as in the
figure below (where $n=5$ and $m=4$).

\begin{figure}
  \includegraphics[height=3.2cm]{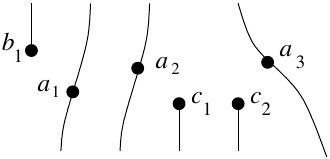}
  \caption{A diagram of long and short strands}
  \label{basis1fig}
\end{figure}

Defining relations are isotopies of these labeled diagrams rel boundary and the relations coming from the algebra $A$ as shown in Figures~\ref{rel1fig}-\ref{rel3fig}. Any floating strand (of the form $cb$ as in Figure~\ref{rel1fig}) reduces to a constant and can be removed by rescaling the
coefficient of the diagram.

\vspace{0.2cm}

\n{\bf A basis for homs.}
We now describe what is obviously a spanning set of $\Hom_{\C}(X^{\otimes n}, X^{\otimes m}).$
Ignoring labels, an isotopy class of a diagram of long and short strands as
in Figure~\ref{basis1fig} corresponds to a partial
order-preserving bijection
\begin{equation}\label{eq-partialB}
f \ : \ [1,n] \ \longrightarrow \ [1,m] .
\end{equation}
Here $[1,n] =\{ 1, \dots, n\}$, viewed as an ordered set with the
standard order. A partial bijection $f: X \longrightarrow Y$ is a
bijection from a subset $L_f$ of $X$ to a subset of $Y$, and
order-preserving means $X,Y$ are ordered sets, and $f(i)<f(j)$ if $i<j$
and $i,j\in L_f$.
Let $L_f=\{i_1, \dots, i_{|f|}\}\subset [1,n]$, where $i_1< i_2< \dots i_{|f|}$.
Here $|f|$ denotes the cardinality of $L_f$,
which we also call the width of
$f$. Denote by $\Bi_{m,n}$ the set of all partial order-preserving bijections
(\ref{eq-partialB}).

A Long strand, counting from left to right, connects the $k$-th element $i_k\in [1,n]$ of
$L_f$, viewed as a point on the bottom edge of the strip, to $f(i_k)\in [1,m]$,
viewed as a point on the top edge. Elements in
$[1,n]\setminus L_f$ are the lower endpoints of the short bottom
strands. Elements of $[1,m]\setminus f(L_f)$ are the upper endpoints of the short top strands.

In Figure~\ref{basis1fig} partial bijection $f:[1,5]\lra [1,4]$ has $L_f=\{ 1, 2, 5\},$ and
$f(1)=2, f(2)=3, f(5)=4$.

\begin{figure}
  \includegraphics[height=4.2cm]{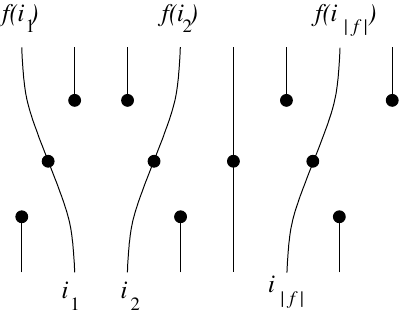}
  \caption{Partial bijection diagram $D_f$}
  \label{basis2fig}
\end{figure}

For each bijection $f$ choose a diagram in the
isotopy class of diagrams representing this bijection and add a dot to each
long strand, see Figure~\ref{basis2fig}. Each short strand already
has a dot at its floating endpoint. Denote this diagram $D_f$.

Partial bijection $f$ has $|f|$ long strands, $n-|f|$
bottom strands and $m-|f|$ top strands.
Let $\BB_f$ be the following set of
elements of $\Hom_{\C}(X^{\otimes n},X^{\otimes m})$.
To the floating endpoint of each bottom strand assign an element of $\BB_{0,1}$ and
denote these elements $c_1, \dots , c_{n-|f|}$ from left to right. To the
floating endpoint of each top strand assign an element of $\BB_{1,0}$ and denote these
elements $b_1, \dots, b_{m-|f|}$ from left to right. To the dot at each
long strand assign an element of $\BB_{1,1}(1)$ (recall that $\BB_{1,1}(1)$ is a
basis of $A''$) and denote them $a_1, \dots, a_{|f|}$. Figure~\ref{basis3fig}
depicts an example with $n=6$, $m=7$ and $|f|=3$.

\begin{figure}
  \includegraphics[height=4.2cm]{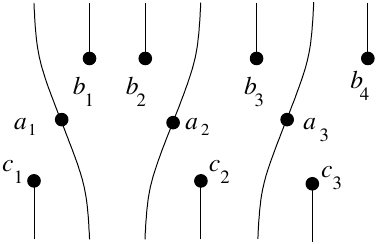}
  \caption{A labeled partial bijection diagram $d$}
  \label{basis3fig}
\end{figure}

The set $\BB_f$ of labeled diagrams for a given $f$ is
naturally parametrized by the set
$$ (\BB_{1,0})^{m-|f|} \times (\BB_{1,1}(1))^{|f|} \times
(\BB_{0,1})^{n-|f|},$$
and each labeled diagram $d$ in $\BB_f$ gives rise
to an element of $\Hom_{\C}(X^{\otimes n},X^{\otimes m})$, also denoted $d$.

\begin{thm}\label{basisthm} The $\K$-vector space $\Hom_{\C}(X^{\otimes n},X^{\otimes m})$ has a basis of labeled diagrams
$$ \BB_{m,n} \ = \ \bigsqcup_{f\in \Bi_{m,n}} \BB_f $$
and is naturally isomorphic to the space
$$ \oplusop{f\in \Bi_{m,n}} ((1-e)Ae)^{\otimes(m-|f|)}\otimes  (A'')^{\otimes |f|} \otimes (eA(1-e))^{\otimes(n-|f|)} .$$
\end{thm}

The notation $\BB_{m,n}$ is compatible for $(m,n)=(1,0), (0,1), (0,0)$ with the notation at the beginning of Section~\ref{basis-section}.

\begin{proof}
We first observe that vector spaces $\K \BB_{m,n}$ admit multiplications
\begin{equation}\label{eq-associative}
 \K \BB_{k,m} \otimes \K \BB_{m,n} \ \lra \ \K \BB_{k,n}
 \end{equation}
that turn the direct sum
\begin{equation}\label{eq-bigRing}
 \K \BB \ := \ \oplusop{n,m\in \N} \K \BB_{m,n}
 \end{equation}
into an idempotented (nonunital) associative algebra, where
$\N =\{0,1,2,\dots\}$.
To compute the product $h_2 h_1 \in \K \BB_{k,n}$ for
$h_2\in \BB_{k,m}$, $h_1\in \BB_{m,n}$ we concatenate the
diagrams $h_2 h_1$ into a single diagram. Every floating strand
in $h_2 h_1$ evaluates to a scalar in $\K$. Long strands in $h_2 h_1$
are concatenations of pairs of long strands in $h_2$, $h_1$, each
carrying a label, say $a_2,a_1 \in \BB_{1,1}(1)$. The concatenation carries the label $a_2a_1 \in (1-e)A(1-e)$ and is simplified as
in Figure~\ref{rel3fig}, with $a=a_2a_1$ on the left hand side.
The right hand side term $a''$ in Figure~\ref{rel3fig}  further decomposes
into a linear combination of elements of $\BB_{1,1}(1)$, and the terms in the sum into
linear combinations of elements of $\BB_{1,0}\times \BB_{0,1}$.

Concatenation of a long strand and a short (top or bottom) strand results in a short (top or bottom)
strand that carries the product label, see the right half of Figure~\ref{rel1fig}. That label is a linear combination of elements in $\BB_{1,0}$, in the top strand case, and elements of $\BB^{0,1}$, in
the bottom strand case.

The simplification procedure is consistent and results in a well-defined element $h_2 h_1$
of $\K \BB_{k,n}$. Associativity of multiplications (\ref{eq-associative}), resulting in well-defined maps
$$ \K\BB_{r,k}\otimes \K \BB_{k,m} \otimes \K \BB_{m,n} \ \lra \ \K \BB_{k,n}$$
for all $r,k,m,n$, follows from the observation that the computation of $h_2 h_1$ can be localized along each concatenation point. Simplification of each pair of strands along their concatenation point can be done independently, and the resulting elements of $A$, interpreted as diagrams, can then be tensored (horizontally composed) to yield $h_2 h_1$.
In this way associativity of (\ref{eq-associative}) follows from associativity of multiplication in  $A$. Since the multiplication is consistent, $\K \BB$ carries an associative non-unital algebra structure.

A substitute for the unit element is a system of idempotents in $\K\BB$.
Denote by $1_n$ the element of $\K \BB_{n,n}$ given by
$n$ parallel vertical lines without labels.
In particular, $1_0=e$.
The diagram $1_n$ is an idempotent corresponding
to the identity endomorphism of $X^{\otimes n}$,
and for any $a\in \BB_{n,m}, b\in \BB_{m,n}$ the products $1_n a  = a$, $b 1_n = b$.
The diagram $1_n$ is the horizontal concatenation of $n$ copies of $1-e\in A$. In all cases considered in this paper $1-e\in \BB_{1,1}(1)$, but this is not necessary in general:  $1-e$ might not be in the basis $\BB_{1,1}(1)$ of $A''$, or even in $A''$.

Elements $1_n$, over all $n\ge 0$, constitute a local system of mutually-orthogonal
idempotents in $\K \BB$. For any finitely many elements $z_1,\dots, z_m$ of $\K\BB$ there
exists $n$ such that $z_i 1_n' = 1_n' z_i = z_i$ for $1\le i \le m$, where
$1_n' = 1_0+1_1+\dots + 1_n$. Non-unital algebra $\K\BB$ can be written as a
direct limit of unital algebras $1_n' \K \BB 1_n'$ under non-unital inclusions
$$ 1_n' \K \BB 1_n' \subset 1_{n+1}' \K \BB 1_{n+1}'.$$
We also refer to an algebra with a local system of idempotents as an \emph{idempotented algebra}.

Multiplication in $\K\BB$ corresponds to vertical concatention of labeled diagrams, and is compatible with the
horizontal concatenation (tensor product) of diagrams, giving us
maps
$$  \K \BB_{m,n} \otimes \K \BB_{m',n'} \ \lra \ \K \BB_{m+m',n+n'}$$
and producing a monoidal category, denoted $\underline{\C}$, with
a single generating object $\underline{X}$ and $\K \BB_{m,n}$ the
space of homs from $\underline{X}^{\otimes n}$ to $\underline{X}^{\otimes m}$. Injectivity of multiplication in
$\underline{\C}$ implies that $\underline{\C}$ is equivalent (and even isomorphic) to $\C$, and that sets $\BB_{m,n}$ are indeed bases of homs in $\C$.

The subalgebra
$$\K\BB_{\le 1} \ = \ \K \BB_{0,0} \oplus \K \BB_{0,1}\oplus \K \BB_{1,0}\oplus \K \BB_{1,1}$$
of $\K \BB$ is naturally isomorphic to $A$.
\end{proof}

Thus, the space $\Hom_{\C}(X^{\otimes n},X^{\otimes m})$ is a direct
sum over all order-preserving partial bijections $f:[1,n]\longrightarrow
[1,m]$ of vector spaces $$ ((1-e)Ae)^{\otimes (m-|f|)} \otimes
(A'')^{\otimes |f|} \otimes (eA(1-e))^{\otimes (n-|f|)}. $$
Denote by $\BB_{m,n}(\ell)$ the subset of $\BB_{m,n}$
corresponding to partial bijections $f$ with $|f|=\ell$:
\begin{align} \label{eq b11}
\BB_{m,n}(\ell) \ = \ \bigsqcup_{|f|=\ell} \BB_f.
\end{align}
These are all basis diagrams with $\ell$ long strands.
The notation $\BB_{m,n}(\ell)$ is compatible with the notation $\BB_{1,1}(1)$ at the beginning of Section \ref{basis-section}.
We have $\BB_{1,1}=\BB_{1,1}(0) \sqcup \BB_{1,1}(1)$, and there is a natural bijection $\BB_{1,1}(0) \cong \BB_{1,0} \times \BB_{0,1}$.

\vspace{0.2cm}

\n {\bf Additive, idempotent complete extension.}
From $\C$ we can form its additive closure $\C^{add}$, with objects---finite directs sums of objects of $\C$.
Category $\C^{add}$ is a $\K$-linear additive strict monoidal category.
Furthermore, let $\Ka(\C)$ be the Karoubi closure of $\C^{add}$.
Category $\Ka(\C)$ is an
idempotent complete $\K$-linear additive strict monoidal category.
Its Grothendieck group $K_0(\Ka(\C))$ is naturally a unital associative ring under
the tensor product operation, and there is a natural homomorphism
$$ \Z[x] \ \longrightarrow K_0(\Ka(\C)) $$
from the ring of polynomials in a variable $x$ to its Grothendieck group,
taking $x$ to $[X]$, the symbol of the generating object $X$ in the Grothendieck
group. The homomorphism may not be injective or surjective.

By an inclusion $\mathcal{A} \subset \B$ of categories we mean a fully
faithful functor $\mathcal{A} \lra \B$.
There is a sequence of categories and inclusion functors
$$ \C \lra \C^{add} \lra \Ka(\C).$$
Category $\C$ is preadditive. Category $\C^{add}$ is additive and
contains $\C$ as a full subcategory. Category $\Ka(\C)$ is additive, idempotent
complete, and contains $\C^{add}$ as a full subcategory. All three categories are
monoidal.

\vspace{0.2cm}

\n {\bf Examples.} We now provide some examples for the above construction.

\n\emph{Example 1:}
A special case of the monoidal category $\C$ appeared in~\cite{KS},
with the
sets $\BB_{1,0}$, $\BB_{0,1}$, and $\BB_{1,1}(1)$ all of cardinality one. Denoting elements of these
sets by $b,$ $c,$ and $(1)$, respectively, the algebra $A$ can be identified with
the Figure~\ref{quiverfig} quiver algebra
subject to the relation $cb=(0)$, where $(j)$, for $j\in \{0,1\}$, denotes the idempotent
path of length zero  at vertex $j$. Thus, the composition $cb$ equals the idempotent path
$(0)$ at the vertex $0$.
Algebra $A$ has a basis $\{ (0),(1),b,c,bc\}$. It's a semisimple algebra isomorphic to
the direct product $\mathrm{M}_2(\K)\times \K$, where the second factor is spanned by
the idempotent $(1)-bc$. The first factor has a basis $\{(0),b,c,bc\}$.

\begin{figure}[h]
  \includegraphics[height=2.0cm]{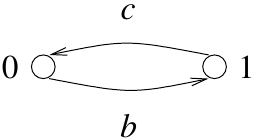}
  \caption{Quiver with vertices $0$, $1$ and arrows $b,c$ between them}
  \label{quiverfig}
\end{figure}

A basis of the hom space
$\Hom_{\C}(X^{\otimes n},X^{\otimes m})$ is given by partial bijections from $[1,n]$ to $[1,m]$,
with no additional decorations necessary. The idempotent completion
$\Ka(\C)$ of the additive closure $\C^{add}$
of $\C$ is semisimple, and $X\cong \mb \oplus X_1$, where the object $X_1=(X,1_X-bc)$ of
the idempotent completion comes
from the above idempotent $1_X-bc$. The simple objects, up to isomorphism, are all tensor
powers of $X_1$.
For this pair $(A,e)$ the natural homomorphism
$ \Z[x] \ \longrightarrow K_0(\Ka(\C)) $ is an isomorphism, see~\cite{KS}.

\vspace{0.1in}

\n\emph{Example 2:} Consider a special case when $1-e\in A'=\mathrm{im}(m').$ Then
$A'= (1-e)A(1-e)$ and $A''=0$. Choose a presentation
$1-e = \sum_{i=1}^n b_i c_i$ with $b_i\in (1-e)Ae,$ $c_i \in eA(1-e)$, and the smallest $n$.
Multiplying this equality by $b_j$ on the right implies $c_i b_j = \delta_{i,j}\in \K$.
Multiplying on the right by any element of $(1-e)Ae$ shows that $b_i$'s are a basis of
$(1-e)Ae$. Likewise, elements $c_i$ are a basis of $eA(1-e)$, and the pair $(A,e)$ is isomorphic
to the pair $(\mathrm{M}_{n+1}(\K), e_{11})$ of a matrix algebra of size $n+1$ and a
minimal idempotent in it. In the additive closure $\C^{add}$ of $\C$ (and in the
idempotent completion $\Ka(\C)$), the object $X$ is isomorphic
to $n$ copies of the unit object $\mb$. This degenerate case is of no interest to us.

Otherwise, $1-e$ is not in the subspace $A'$ of $(1-e)A(1-e)$,
and we can always choose $A''$ and $\BB_{1,1}(1)$ to contain
$1-e$, ensuring that the vertical line diagram is in the basis $\BB_{1,1}(1)$.

\vspace{0.05in}

\n{\bf Case when $A$ is a super algebra.}
We now discuss a generalization when $A$ is an algebra in the category of super-vector spaces. In that category the objects are $\Z/2$-graded and degree one summands are called odd components. Algebra $A$ must be $\Z/2$-graded, $A=A_0\oplus A_1$, with the idempotent $e\in A_0$ such that $eAe=\K$.
Then $1-e$ is also in $A_0$. Vector spaces $eA(1-e)$, $(1-e)Ae$,
$eAe$, and $(1-e)A(1-e)$ are then each a direct sum of its homogeneous components. For instance $eA(1-e) = eA_0(1-e) \oplus eA_1(1-e),$ with
$eA_i(1-e)$ being the degree $i$ component of $eA(1-e)$ for $i=0,1$.

We continue to require injectivity of $m'$. Subspace $A'=\mathrm{im}(m')$ is
$\Z/2$-graded, and we select its complement $A''$ to be graded as well.
All basis elements of $\BB_{0,0}, \BB_{0,1}, \BB_{1,0}, \BB_{1,1}(1)$ should be homogeneous (which is
automatic for $\BB_{0,0}$). This will result in a $\Z/2$-graded idempotented
algebra $\K\BB$ with homogeneous basis elements in $\BB$.

The monoidal category $\C$ that we assign to $(A,e)$ is enriched over the
category of super-vector spaces. For this reason, homogeneous morphisms
in $\C$ supercommute when their relative height order changes during an
isotopy, see
Figure~\ref{super2fig}, with the coefficient $(-1)^{|a|\cdot |b|}$, where
$|a|\in \{0,1\}$ is the $\Z/2$-degree of the generator $a$.

\begin{figure}[h]
\includegraphics[height=2.3cm]{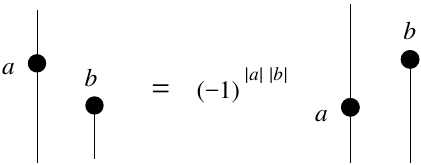}
\caption{Super-commutativity of morphisms}
\label{super2fig}
\end{figure}


In our construction of a basis of hom spaces
$\Hom_{\C}(X^{\otimes n},X^{\otimes m})=\K \BB_{m,n}$,
 we need to choose a particular order of
  heights when dealing with decorated basis diagrams, to avoid sign
  indeterminancy.
  The order that we follow is shown in Figure~\ref{super1fig}.
  Lengths of top short arcs increase going from left to right, so that
  $b_1$ label is at the highest position, followed by $b_2$, and so on.
  Highest long strand label $a_1$ is below the lowest $b$-label $b_{m-|f|}$.
  It's followed by $a_2$ to the right and below, all the way to $a_{|f|}$, which
  has the lowest height of all $a$ labels.
  Leftmost bottom arc label $c_1$ is lower than $a_{|f|}$ label, and the
  remaining bottom arc labels
  $c_2$, \dots, $c_{n-|f|}$ continue with the lower heights. The lowest label
  in the diagram is $c_{n-|f|}$.

\begin{figure}[h]
\includegraphics[height=5cm]{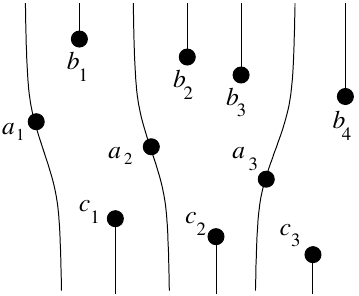}
\caption{Keeping track of heights of labels,  left to right and top to bottom}
\label{super1fig}
\end{figure}

Proof of Theorem~\ref{basisthm} extends without any changes, simply by confirming the consistency of signs in several places. The theorem implies that there is no collapse in the size of homs between tensor powers of $X$ and gives a basis
for the space of morphisms from $X^{\otimes n}$ to $X^{\otimes m}$.

\vspace{0.05in}

\n {\bf Case when $A$ is a DG algebra with the trivial differential.}
A mild generalization of our construction to DG (differential graded) algebras will be needed in Section~\ref{dg-section}. A DG algebra $A$ is a $\Z$-graded algebra, $A=\oplusop{i\in \Z} A^i$,
with a differential $d$ of degree one, such that $d(ab) = d(a)b+(-1)^{|a|}ad(b),$ where $|a|\in \Z$ is the degree of $a$, for homogeneous $a$. Each DG algebra is viewed as a superalgebra by reducing the degree modulo two and forgetting the differential.

In this paper we will only encounter the simplest case, when
the differential is trivial on $A$. We assume this to be the
case. With the differential zero on $A$, no additional conditions
on $A$ or $e$ are needed, and the generalization from the super algebras to such DG algebras is straightforward. The grading is
now by $\Z$ and not just $\Z/2$, with idempotent $e$ in degree zero.

We choose $A''$ to be a subspace that is the direct sum of its intersections with the homogenous summands of $(1-e)A(1-e)$. Likewise, bases $\BB_{0,0}, \BB_{0,1}, \BB_{1,0}, \BB_{1,1}(1)$ are chosen to consist of homogeneous elements.

The DG category $\C$ is constructed from this data just as in the
super-algebra case. Due to $\Z$-grading, one introduces enlarged
morphism spaces,
 $\HOM_{\C}(M,N) = \bigoplus\limits_{m\in \Z}\Hom_{\C}(M,N[m]).$

Morphism spaces $\HOM_{\C}$  between tensor powers of $X$ have bases as described in Theorem~\ref{basisthm}, with heights in the basis diagrams tracked as in Figure~\ref{super1fig}. The differential acts by zero
on all morphism spaces $\HOM_{\C}(X^{\otimes n}, X^{\otimes m})$.
Diagrams super-commute, with the super grading given by reducing
the $\Z$-grading modulo two.

\vspace{0.1in}

\n {\bf A chain of ideals $J_{n,k}$.} Now assume
$A$ is an algebra, or a super-algebra, or a DG algebra with the trivial
differential. The ring $A_k=\End_{\C}(X^{\otimes k})$
is spanned by diagrams of decorated long and short strands, with each diagram
having $\ell$ long strands and $2(k-\ell)$ short strands, an equal number $k-\ell$
at both top and bottom. Composing two such diagram $D_1$, $D_2$
with $\ell_1$ and $\ell_2$ long arcs, correspondingly, results in the
product $D_2D_1$, which is also an endomorphism of $X^{\otimes k}$, that
decomposes into a linear combination of diagrams, each with at most
$\min(\ell_1,\ell_2)$ long strands. The number of long strands in a diagram
cannot increase upon composition with another diagram.

Therefore, there is a two-sided ideal $J_{n,k}$ of $A_k$ whose elements are
linear combinations of diagrams with at most $n$ long strands. Here $0\le n \le k$. It's also convenient to define $J_{-1,k}$ to be the zero ideal.
There is a chain of inclusions of two-sided ideals
\begin{gather} \label{def ideal}
0 = J_{-1,k}\subset J_{0,k}\subset \dots \subset J_{k-1,k}\subset J_{k,k} = A_k.
\end{gather}
Basis $\BB_{k,k}$ of $A_k$ respects this ideal filtration, and
restricts to a basis in each $J_{n,k}$.
\begin{cor} \label{cor-basis-ideal}
Two-sided ideal $J_{n,k}$ has a basis
$$\mathbb{B}_{k,k}(\le n) \ = \  \bigsqcup\limits_{0 \leq \ell \leq n} \mathbb{B}_{k,k}(\ell).$$
\end{cor}

In the special case $k=1$, we denote by $J$ the ideal $J_{0,1}$ and
by $L$ the quotient ring $A_1/J_{0,1}$. Thus, there is an
exact sequence
$$ 0 \lra J \lra A_1 \lra L \lra 0,$$
where
\begin{eqnarray*}
   J & = & J_{0,1} = (1-e)AeA(1-e)  \\
   A_1 & = & (1-e)A(1-e) \\
   L  & = & A_1/J \cong (1-e)A(1-e)/(1-e)AeA(1-e)
   \end{eqnarray*}
$A_1=\End_{\C}(X)$ is a subring of $A$, with the unit element $1-e$,
and $L$ is isomorphic to the quotient of $A_1$ by the two-sided ideal of maps that factors
through $\mb$.

The quotient ring
\begin{gather} \label{def lk}
L_k = A_k / J_{k-1,k}
\end{gather}
is naturally isomorphic
to $L^{\otimes k}$, the $k$-th tensor power of $L$.
Graphically, we quotient the space of linear combinations of
decorated diagrams with $k$ endpoints at both bottom and top by
the ideal of diagrams with at least one short strand (necessarily at
least one at the top and the bottom). Elements in the quotient by
this ideal will be represented by linear combinations of diagrams
of $k$ decorated long strands, modulo diagrams where a long strand
simplifies into a linear combination of a pair of decorated short
strands. The quotient is isomorphic to $L=L_1$, defined above for $k=1$,
and to the $k$-th tensor power of $L$ for general $k$.
In the super-case, the tensor power is understood correspondingly,
counting signs.


\section{A categorification of $\Ztwo$}
\label{dg-section}

The goal of this section is to describe a monoidal DG category $\C$, and its associated monoidal triangulated Karoubi closed category $D^c(\C)$.
We show that the Grothendieck ring $K_0(D^c(\C))$ is isomorphic to $\Ztwo$.

\subsection{A diagrammatic category $\C$}
As before, we work over a field  $\K$.
Consider a pre-additive monoidal category $\C$ with one generating object $X$, enriched over the category
of $\Z$-graded super-vector spaces over $\K$, with the supergrading given by reducing the $\Z$-grading modulo $2$.
In a pre-additive category, homomorphisms between any two objects constitute an abelian group (in our case, a
$\Z$-graded $\K$-vector space), but direct sums of objects are not formed.

A set of generating morphisms together with their degrees
is given in Figure \ref{sec3fig1}.

  \begin{figure}[h]
    \includegraphics[height=4cm]{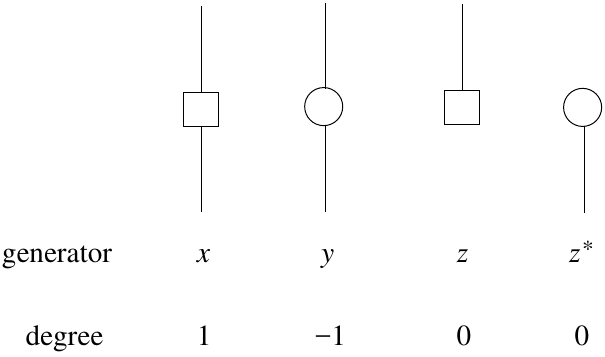}
    \caption{Generating morphisms}
  \label{sec3fig1}
   \end{figure}

Two of these four generating morphisms
are endomorphisms of $X$, of degrees $1$ and $-1$, correspondingly,
one is a morphism from $\mb$ to $X$ of degree $0$, and the fourth morphism
goes from $X$ to $\mb$ and has degree $0$. We denote these generators
$x,y,z,\wtz$, from left to right, so that $x,y$ are endomorphisms of $X$, $z$ a morphism from $\mb$ to $X$, and $\wtz$ a morphism from $X$ to $\mb$.
We draw $x$ as a long strand decorated by a box, $y$ as a long strand decorated by a circle, $z$ as a short top strand decorated by a box, and $\wtz$ as a short bottom strand decorated by a circle, respectively.
A pair of far away generators super-commute.
The first two generators $x$ and $y$ have odd degrees, while $z$ and $\wtz$ have even degrees.

Local relations are given in Figure~\ref{sec3fig2}.
They are
\begin{align} \label{local rel}
\begin{split}
 & z^* z= 1_{\mb}, \ \ \ \ z^* x = 0, \ \ \ \  y z=0, \\
 & yx= 1_X, \ \ \ \  xy + zz^* = 1_X.
\end{split}
\end{align}
The identity map $1_{\mb}$ of the object $\mb$ is represented by the empty diagram.
Figure~\ref{sec3fig3} shows our notation for powers and some compositions of the generating morphisms.


  \begin{figure}[h]
    \includegraphics[height=6cm]{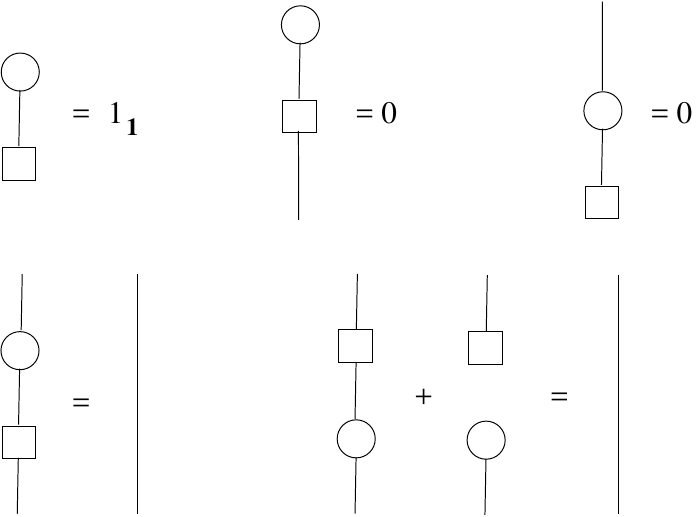}
    \caption{Defining local relations.}
  \label{sec3fig2}
   \end{figure}

  \begin{figure}[h]
    \includegraphics[height=3.4cm]{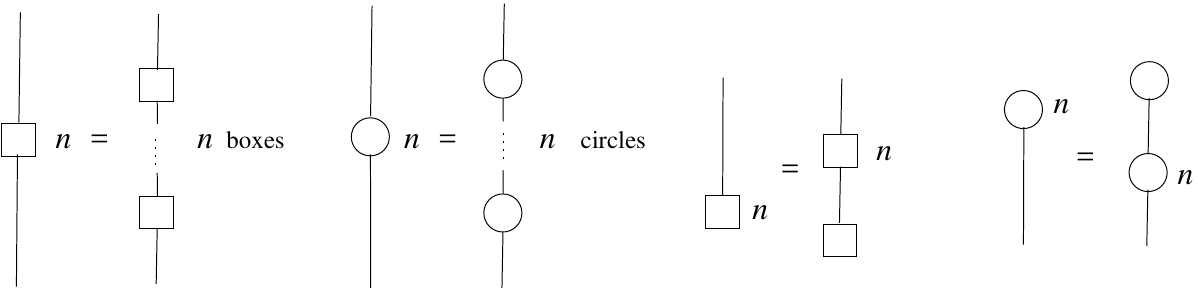}
    \caption{Notations for compositions, left to right: $x^n, y^n, x^nz$ and $\wtz y^n$.}
  \label{sec3fig3}
   \end{figure}

We write $\Hom_{\C}(M,N)$ for the vector space of degree $0$ morphisms,
and $\HOM_{\C}(M,N)$ for the graded vector space with degree components--homogeneous maps of degree $m$:
$$ \HOM_{\C}(M,N) = \bigoplus\limits_{m\in \Z}\Hom_{\C}(M,N[m]).$$

If $\op{char}(\K)\neq 2$ we choose an order of heights of decorations as follows.
For any pair of strands, the height of decorations on the left strand is above the height of decorations on the right strand.
Let $f \ot g$ denote the horizontal composition of two diagrams $f$ and $g$, where the height of $f$ is above that of $g$, see Figure~\ref{sec3fig4}.

\begin{figure}[h]
    \includegraphics[height=2.6cm]{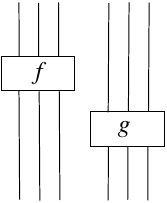}
    \caption{Convention for $f\ot g$}
    \label{sec3fig4}
\end{figure}

\vspace{.2cm}
\n{\bf Bases of morphism spaces.}
We observe that the category $\C$ is generated from the suitable
data $(A,e)$ as described in Section~\ref{basis-section},
where $A$ is a DG algebra with the trivial differential.
To see this, we restrict the above diagrams and defining relations on them to the case when there is at most one strand at the top and at most
one strand at the bottom. In other words, we consider generating
morphisms in Figure~\ref{sec3fig1} and compose them only vertically,
not horizontally, with the defining relations in Figure~\ref{sec3fig2}. The idempotent $e$ is given by the
empty diagram, while $1-e$ is the undecorated vertical strand diagram.

Ignoring the grading and the (zero) differential on $A$, it follows from the defining relations (\ref{local rel}) that $A$ is isomorphic to the Jacobson algebra~\cite{J} and to the Leavitt path algebra $L(T)$ of the Toeplitz graph $T$ in (\ref{quiver2}), see~\cite[Example 7]{A}.
In particular, as a $\K$-vector space, algebra $A$ has a basis
$$ \{1_{\mb}\} \cup \{ x^n z ~|~ n\ge 0\} \cup \{\wtz y^n ~|~ n \ge 0\} \cup \{ x^n z \wtz y^m ~|~ n,m\ge 0 \} \cup \{1_X, x^n, y^n ~|~ n >0 \} $$
by~\cite[Corollary 1.5.12]{AAS} or by a straightforward computation.
Our notations for some of these basis elements are shown in
 Figure~\ref{sec3fig3}.

The basis of $A$ can be split into the following disjoint subsets:
\be
\item $eAe\cong \K$ has a basis  $\mathbb{B}_{0,0}=\{1_{\mb}\}$ consisting of a single element which is the empty diagram;
\item $(1-e)Ae$ has a basis $\mathbb{B}_{1,0}=\{x^n z ~|~ n \geq 0\}$. Element $x^n z$
is depicted by a short top strand decorated by a box with label $n$, see Figure~\ref{sec3fig3};
\item $eA(1-e)$ has a basis $\mathbb{B}_{0,1}=\{\wtz y^n ~|~ n\geq 0\}$. Element
$\wtz y^n$ is depicted by a short bottom strand decorated by a circle with label $n$ (lollipop in Figure~\ref{sec3fig3});
\item $(1-e)A(1-e)$ has a basis $\mathbb{B}_{1,1}(0) \sqcup \mathbb{B}_{1,1}(1)$, where $\mathbb{B}_{1,1}(0)=\{x^n z \wtz y^m ~|~ n,m \geq 0\}$ consists of pairs (short top strand with a labelled box, short bottom strand with a labelled circle), and $\mathbb{B}_{1,1}(1)=\{1_X, x^n, y^n ~|~ n >0 \}$ consists of long strand diagrams which may carry either circles or boxes, but not both.
\ee
The multiplication map $ (1-e)Ae \otimes eA(1-e)  \ra (1-e)A(1-e)$ sends the basis $\mathbb{B}_{1,0} \times \mathbb{B}_{0,1}$ of
$(1-e)Ae \otimes eA(1-e)$ bijectively to $\mathbb{B}_{1,1}(0)$ so that the multiplication map is injective.

We see that the conditions on $(A,e)$ from
the beginning of Section~\ref{basis-section} are satisfied,
and we can indeed form the monoidal category $\C$ as above with objects
$X^{\otimes n}$, over $n\ge 0$.
Algebra $A$ can then be described as the direct sum
$$A\ \cong \ \END_{\C}(\mb)\oplus \HOM_{\C}(\mb, X) \oplus \HOM_{\C}(X,\mb) \oplus \END_{\C}(X),$$
which is a DG algebra with the trivial differential.
Therefore, a basis of $\HOM_{\C}(\xn, \xm)$ is given in Theorem \ref{basisthm}.



\subsection{DG extensions of $\C$}
\label{subsec-DG-ext}

We turn the category $\C$ into a DG category by introducing a differential
$\partial$ which is trivial on all generating morphisms. Necessarily,
$\partial$ is trivial on the space of morphisms between any two objects of $\C$.
The resulting DG category is still denoted by $\C$.

We refer the reader to \cite{Ke2} for an introduction to DG categories.
For any DG category $\D$ we write $\Hom_{\D}(Y,Y')$ for the vector space of degree $0$ morphisms,
and $\HOM_{\D}(Y,Y')$ for the chain complex of vector spaces with degree components $\Hom_{\D}(Y,Y'[m])$ of homogeneous maps of degree $m$.
A {\em right DG $\D$-module} $M$ is a DG functor $M: \D^{op} \ra Ch(\K)$ from the opposite
DG category $\D^{op}$ to the DG category of chain complexes of $\K$-vector spaces.
For each object $Y$ of $\D$, there is a right module $Y^{\we}$ {\em represented by} $Y$
$$Y^{\we}=\HOM_{\D}(-,Y).$$
Unless specified otherwise, all DG modules are right DG modules in this paper.

We use the notations from \cite[Section 3.2.21]{Sch2}.
For any DG category $\D$, there is a canonical embedding $\D \subset \wtd$ of $\D$ into the pre-triangulated DG category $\wtd$ associated to $\D$.
It's obtained from $\D$ by formally adding iterated shifts, finite direct sums, and cones of morphisms.
The homotopy category $\op{Ho}(\wtd)$ of $\wtd$ is triangulated.
It is equivalent to the full triangulated subcategory of the derived category $D(\D)$ of DG $\D$-modules which is generated by $\D$.
Each object $Y$ of $\op{Ho}(\wtd)$ corresponds to a module $Y^\we$ of $D(\D)$ under the equivalence.
The idempotent completion $\wt{\op{Ho}}(\wtd)$ of $\op{Ho}(\wtd)$ is equivalent to the triangulated category $D^c(\D)$ of compact objects in $D(\D)$ by \cite[Lemma 2.2]{Nee}.

To summarize, there is a chain of categories
$$\D \subset \wtd \dashrightarrow \op{Ho}(\wtd) \subset \wt{\op{Ho}}(\wtd) \simeq D^c(\D) \subset D(\D).$$
The first two categories are DG categories, and $\D \subset \wtd$ is fully faithful.
The last four categories are triangulated.
The dashed arrow between $\wtd$ and $\op{Ho}(\wtd)$ is not a functor.
More precisely, $\op{Ho}(\wtd)$ has the same objects as $\wtd$, and morphism spaces as subquotients of morphism spaces of $\wtd$.
It is a full subcategory of its idempotent completion $\wt{\op{Ho}}(\wtd)$.
The category $D^c(\D)$ of compact objects in $D(\D)$ is a full triangulated subcategory of $D(\D)$.

\begin{defn} \label{def B(A)}
For a unital DG algebra $R$, let $\B(R)$ be a DG category with a single object $\ast$ such that $\END_{\B(R)}(\ast)=R$.
Let $D(R)$ and $D^c(R)$ denote $D(\B(R))$ and $D^c(\B(R))$, respectively.
\end{defn}

\begin{rmk} \label{rmk DA}
If $R$ is an ordinary unital algebra viewed as a DG algebra concentrated in degree $0$ with the trivial differential, then $D(\B(R))$ is equivalent to the derived category of $R$-modules, and $D^c(\B(R))$ is equivalent to the triangulated category of perfect complexes of $R$-modules, see~\cite[Section 6.5]{Kr} and~\cite[Proposition 70.3]{St}.
\end{rmk}

Since the DG category $\C$ is monoidal, it induces a monoidal structure on $\wtc$ which preserves homotopy equivalences.
There are induced monoidal structures on the triangulated categories $\op{Ho}(\wtc)$ and $\wt{\op{Ho}}(\wtc)$.
We are interested in the Grothendieck ring of $\wt{\op{Ho}}(\wtc) \simeq D^c(\C)$.

\vspace{.2cm}
\n{\bf Isomorphisms in $D^c(\C)$:}
Each morphism $f \in \Hom_{\C}(Y,Y')$ with $\bdry f=0$ induces a morphism in $\Hom_{D^{c}(\C)}(Y^\we,{Y'}^\we)$, denoted $f$ by abuse of notation.
The generating morphisms in Figure~\ref{sec3fig1} and the local relations in Figure~\ref{sec3fig2} induce an isomorphism in $D^c(\C)$
\begin{gather} \label{eq iso}
X^\we \cong \mb^\we \oplus X^\we[-1],
\end{gather}
given by $(\wtz, y)^T \in \Hom_{D^c(\C)}(X^\we, \mb^\we \oplus X^\we[-1])$, and $(z,x) \in \Hom_{D^c(\C)}(\mb^\we \oplus X^\we[-1], X^\we)$, see Figure~\ref{sec3fig5}.
Tensoring with $(X^\we)^{\otimes (k-1)}$ in $D^c(\C)$ on either side of isomorphism (\ref{eq iso}) results in isomorphisms in $D^c(\C)$
\begin{gather} \label{eq iso2}
(X^\we)^{\otimes k} \cong (X^\we)^{\otimes (k-1)} \oplus (X^\we)^{\otimes k}[-1].
\end{gather}


  \begin{figure}[h]
    \includegraphics[height=4cm]{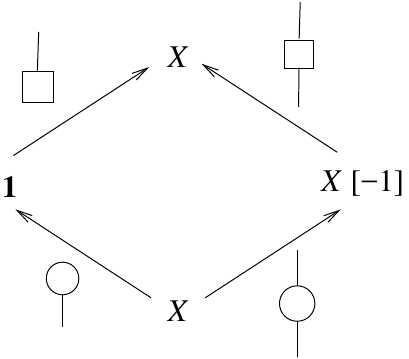}
    \caption{The isomorphism $X^\we \cong \mb^\we \oplus X^\we[-1]$}
    \label{sec3fig5}
  \end{figure}

\subsection{DG algebras of endomorphisms} \label{subsec dga of endo}
Part of the structure of $\C$ can be encoded into an idempotented DG algebra $B$ with the trivial differential, which has a complete system of mutually orthogonal idempotents $\{1_k\}_{k\ge 0},$ so that
$$ B \ = \ \bigoplus\limits_{m,n\ge 0} 1_m B 1_n,$$
and
$$ 1_m B 1_n = \HOM_{\C}(X^{\otimes n},X^{\otimes m}).$$
Multiplication in $B$ matches composition of morphisms in $\C$.

The tensor structure of $\C$ induces a tensor structure on $B$.
Given $f,g \in B$ represented by some diagrams in $\C$, let $f \ot g$ be an element of $B$ represented by the horizontal composition of the two diagrams for $f$ and $g$, where the diagram for $f$ is on the left whose height is above the height of the diagram for $g$.
There is the super-commutativity relation
$$(f \ot g)(f' \ot g')=(-1)^{\op{deg}(g)\op{deg}(f')}ff' \ot gg'$$
for homogeneous elements $f,f',g,g' \in B$.

We also define
$$ B_k = \bigoplus\limits_{m,n\le k} 1_m B 1_n,$$
which is a DG algebra with the trivial differential and the unit element $\sum\limits_{n \le k} 1_n$.
The inclusions $B_k \subset B_{k+1}$ and $B_k \subset B$ are nonunital.
Define
\begin{gather} \label{def An}
A_k = 1_k B 1_k = \op{END}_{\C}(\xk),
\end{gather}
which is a DG algebra with the trivial differential and the unit element $1_k$.
For $k=0$, the DG algebras
\begin{gather} \label{def A0}
A_0=B_0 \cong \K.
\end{gather}
The inclusion $A_k \subset B_k$ is nonunital for $k>0$.

Let $\alpha_k: A_{k-1} \hookrightarrow A_{k}$ be an inclusion of DG algebras given by tensoring with $z\wtz$ on the left
\begin{gather} \label{def alpha}
\alpha_k(f)=(z\wtz)\otimes f,
\end{gather}
for $f \in A_{k-1}$.
Note that $\alpha_k$ is nonunital.

\begin{defn} \label{def Jk}
For $k\geq 1$, let $J_{k}$ be the two-sided DG ideal of $A_k$ generated by diagrams with at most $k-1$ long strands.
The quotient $L_k=A_k / J_{k}$ is naturally a unital DG algebra with the trivial differential.
\end{defn}

By Theorem~\ref{basisthm} and Proposition~\ref{cor-basis-ideal} in Section~\ref{basis-section}, $A_k$ has a $\K$-basis $\bigsqcup\limits_{0\leq \ell \leq k}\mathbb{B}_{k,k}(\ell)$, and $J_{k}$ has a $\K$-basis $\bigsqcup\limits_{0\leq \ell \leq k-1}\mathbb{B}_{k,k}(\ell)$.
So $L_k$ has a $\K$-basis given by the images of elements of $\mathbb{B}_{k,k}(k)$ under the quotient map $A_k \ra L_k$.

The ideal $J=J_{1}$ has a $\K$-basis $\mathbb{B}_{1,1}(0)=\{x^iz\wtz y^j ~|~ i,j\geq0\}$.
The unital DG algebra $L=L_1$ is generated by $\overline{x}, \overline{y},$ which are the images of $x, y \in A_1$ under the quotient map $A_1 \ra L$.
There is an exact sequence
$$0 \ra J \ra A_1 \ra L \ra 0$$
of DG algebras with the trivial differentials.
For $n\ge 0$, let $\mbox{M}_n(\K)$ be the $(n+1)\times (n+1)$
matrix DG algebra with the trivial differential and a
standard basis $\{e_{ij} ~|~ 0 \leq i,j \leq n\}$ of elementary matrices, with $\mbox{deg}(e_{ij})=i-j$.

\begin{prop} \label{prop iso}
There are isomorphisms of DG $\K$-algebras with trivial differentials:
\begin{gather*}
J \cong \mbox{M}_{\N}( \K),\\
L \cong \K[a,a^{-1}], \quad \mbox{deg}(a)=1,\\
\lk \cong \K\lan a_1^{\pm1}, \dots, a_k^{\pm1} \ran / (a_ia_j=-a_ja_i, i \neq j), \quad \mbox{deg}(a_i)=1.
\end{gather*}
\end{prop}
\begin{proof}
Define the isomorphism $\mbox{M}_{\N}( \K) \ra J$ by $e_{ij} \mapsto x^iz\wtz y^j$ for $i,j \in \N$. The nonunital DG algebra $J$ is isomorphic to the direct limit $\mbox{M}_{\N}(\K)$ of unital DG algebras $\mbox{M}_n( \K)$
under non-unital inclusions $\mbox{M}_n( \K)\subset
\mbox{M}_{n+1}(\K)$ taking $e_{ij}$ to $e_{ij}$.

Define a map of algebras $L \ra \K[a,a^{-1}]$ by $\overline{x} \mapsto a, \overline{y} \mapsto a^{-1}.$
It is an isomorphism since
$$\overline{y}~\overline{x}=\overline{yx}=1 \in L, \qquad \overline{x}~\overline{y}=\overline{xy}=\overline{1-z\wtz}=1 \in L,$$
by the local relations (\ref{local rel}).

For $1 \leq i \leq k$, let $x_i=1 \ot \cdots \ot x \ot \cdots 1$ and $y_i=1 \ot \cdots \ot y \ot \cdots 1 \in A_k$ whose $i$th factors are $x$ and $y \in A_1$, respectively.
Then $\lk$ is generated by images $\overline{x_i}, \overline{y_i}$, and subject to relations
$$\overline{x_i} ~ \overline{y_i}=\overline{y_i} ~ \overline{x_i}=1, \qquad \overline{x_i} ~ \overline{x_j}=-\overline{x_j} ~ \overline{x_i} \quad \mbox{for}~~i\neq j.$$
Define the isomorphism $\lk \ra \K\lan a_1^{\pm1}, \dots, a_k^{\pm1} \ran / (a_ia_j=-a_ja_i, i \neq j)$ by $\overline{x_i} \mapsto a_i, \overline{y_i} \mapsto a_i^{-1}$ for $1 \leq i \leq k$.
\end{proof}

  \begin{figure}[h]
    \includegraphics[height=4cm]{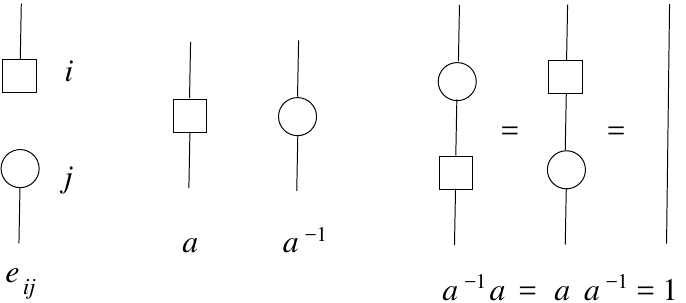}
    \caption{Basis element $e_{ij}$ of the ideal $J$ and elements $a,a^{-1}$ of the quotient $L$, where generators of $L$ are represented by the same diagrams as for $A_1$ by abuse of notation.}
    \label{sec3fig6}
  \end{figure}

We fix the isomorphisms for $J, L, L_k$ in Proposition \ref{prop iso}.
See Figure~\ref{sec3fig6} for $J$ and $L$.

\subsection{Approximation of $D^c(\C)$ by $D^c(A_k)$.}

Let $\C_k$ be the smallest full DG subcategory of $\C$ which contains the objects $X^{\otimes n}$, $0 \leq n \leq k$.
Note that $\C_k$ is not monoidal.
There is a family of inclusions $\C_{k-1} \subset \C_k$ of DG categories.
They induce a family of functors $\imath_k: D(\C_{k-1}) \ra D(\C_k)$.
For $0 \leq n \leq k-1$, $\imath_k(\xnw)=\xnw$ is compact in $D(\C_k)$, and
$$\op{END}_{D(\C_{k-1})}(\xnw) \cong A_{n} \cong \op{END}_{D(\C_{k})}(\xnw) \cong \op{END}_{D(\C_{k})}(\imath_k(\xnw)).$$
The functor $\imath_k: D(\C_{k-1}) \ra D(\C_k)$ is fully faithful by \cite[Lemma 4.2 (a, b)]{Ke1}.
The restriction to the subcategory $\imath_k^c: D^c(\C_{k-1}) \ra D^c(\C_k)$ of compact objects is also fully faithful.
Similarly, there is a family of inclusions $g_k^c: D^c(\C_{k}) \ra D^c(\C)$ of triangulated categories.

Recall that the Grothendieck group $K_0(\T)$ of an essentially small triangulated category $\T$ is the abelian group generated by symbols $[Y]$ for every object $Y$ of $\T$, modulo the relation $[Y_2]=[Y_1]+[Y_3]$ for every distinguished triangle $Y_1 \ra Y_2 \ra Y_3 \ra Y_1[1]$ in $\T$.
In particular, $[Y_1]=[Y_2]$ if $Y_1$ and $Y_2$ are isomorphic.
The functors $\imath_k^c$ and $g_k^c$ send distinguished triangles to distinguished triangles and induce maps of abelian groups
$${\imath_k^c}_*: K_0(D^c(\C_{k-1})) \ra K_0(D^c(\C_{k})), \qquad {g_k^c}_*: K_0(D^c(\C_{k})) \ra K_0(D^c(\C)).$$
Let $\varinjlim K_0(D^c(\C_k))$ denote the direct limit of $K_0(D^c(\C_{k}))$ with respect to ${\imath_k^c}_*$.

\begin{prop} \label{prop K0 of C}
There is an isomorphism of abelian groups
$$K_0(D^c(\C)) \cong \varinjlim K_0(D^c(\C_k)).$$
\end{prop}
\begin{proof}
The family of maps ${g_k^c}_*$ induces a map $g_*: \varinjlim K_0(D^c(\C_k)) \ra K_0(D^c(\C))$ since functors $g_{k-1}^c$ and $g_{k}^c \circ \imath_k^c$ are isomorphic.
The map $g_*$ is surjective since any object of $D^c(\C)$ is contained in $D^c(\C_k)$ for some $k$ up to isomorphism.
The map $g_*$ is injective since any distinguished triangle in $D^c(\C)$ is contained in $D^c(\C_k)$ for some $k$ up to isomorphism.
\end{proof}

The category $\C_k$ contains a full DG subcategory $\C_k'$ of a single object $\xk$ whose endomorphism DG algebra $\op{END}_{\C_k}(\xk)=A_k$ by (\ref{def An}).
Thus, the category $\C_k'$ is isomorphic to $\B(A_k)$, and $D(A_k)=D(\B(A_k))$, see Definition~\ref{def B(A)}.
There is an inclusion $\B(A_k) \subset \C_k$ of DG categories.
The induced functors
\begin{align} \label{hkc}
h_k: D(A_k) \ra D(\C_k), \qquad h_k^c: D^c(A_k) \ra D^c(\C_k)
\end{align}
of triangulated categories are fully faithful by \cite[Lemma 4.2 (a, b)]{Ke1}.
A set $\cal{H}$ of objects of a triangulated category $\T$ is a {\em set of generators} if $\T$ coincides with its smallest strictly full triangulated subcategory containing $\cal{H}$ and closed under infinite direct sums, see \cite[Section 4.2]{Ke1}.
In particular, $\{\xnw, 0 \leq n \leq k\}$ forms a set of generators for $D(\C_k)$.
Equation (\ref{eq iso2}) implies that $\xnw$ is isomorphic to a direct summand of $\xkw$ for $0 \leq n \leq k$.
Let $p_n \in \op{End}_{D(\C_k)}(\xkw)$ denote the idempotent of projection onto the direct summand $\xnw$.
Then $\xnw$ is isomorphic to a DG $\C_k$-module given by a complex
$$\cdots \xra{1-p_n} \xkw \xra{p_n} \xkw \xra{1-p_n} \xkw.$$
Thus, $\{\xkw\}$ forms a set of compact generators for $D(\C_k)$.
The functor $h_k$ is an equivalence of triangulated categories by \cite[Lemma 4.2 (c)]{Ke1}.
It is clear that $h_k^c: D^c(A_k) \ra D^c(\C_k)$ is also an equivalence
and thus induces an isomorphism of Grothendieck groups
 ${h_k^c}_*: K_0(D^c(A_k)) \cong K_0(D^c(\C_k))$.
By Proposition~\ref{prop K0 of C}, there is a canonical isomorphism of abelian groups:
\begin{gather} \label{iso K0 of C}
K_0(D^c(\C)) \cong \varinjlim K_0(D^c(A_k)).
\end{gather}


\subsection{K-theory computations} \label{Sec K}
For a DG category $\D$, let $K_0(\D)$ denote the Grothendieck group of the triangulated category $D^c(\D)$.

If $R$ is an ordinary unital algebra viewed as a DG algebra concentrated in degree $0$ with the trivial differential, then there is a canonical isomorphism $K_0(D^c(R)) \cong K_0(R)$ by Remark \ref{rmk DA}, where $K_0(R)$ is the Grothendieck group of the ring $R$.

Without ambiguity let $K_0(R)$ denote $K_0(D^c(R))$ for a unital DG algebra $R$.
The isomorphism (\ref{iso K0 of C}) can be rewritten as
$$K_0(\C) \cong \varinjlim K_0(A_k).$$

In order to compute $K_0(A_k)$, we need higher K-theory of DG algebras and DG categories.
We briefly recall the definition of higher K-theory of DG categories following~\cite[Section 3.2.21]{Sch2}.
Schlichting~\cite[Section 3.2.12]{Sch2} introduces the notion of {\em complicial exact category with weak equivalences} whose higher K-theory is defined.
For a DG category $\D$, its pre-triangulated envelope $\wtd$ can be made into an exact category whose morphisms are maps of degree $0$ which commute with the differential. A sequence is exact if it is a split exact sequence when ignoring the differential.
 Then $(\wtd, w)=(\wtd, \mbox{homotopy equivalences})$ is a complicial exact category with homotopy equivalences as weak equivalences.
The K-theory of the DG category $\D$ is defined as the K-theory of the complicial exact category with weak equivalences $(\wtd, w)$.
This definition is equivalent to Waldhausen's definition of K-theory of a DG category according to \cite[Remark 3.2.13]{Sch2}.

We introduce the following notations.
For a DG category $\D$,
\begin{gather} \label{notation1}
K_1(\D)=K_1(\wtd, w), \qquad K_0'(\D)=K_0(\wtd, w).
\end{gather}
For a unital DG algebra $A$,
\begin{gather} \label{notation2}
K_1(A)=K_1(\B(A)), \qquad K_0'(A)=K_0'(\B(A)).
\end{gather}
Note that $K_0'(\D) \cong K_0(\op{Ho}(\wtd))$ by \cite[Proposition 3.2.22]{Sch2}.
Recall that $$K_0(\D)=K_0(D^c(\D)) \cong K_0(\wt{\op{Ho}}(\wtd)).$$
By \cite[Corollary 2.3]{Th}, $K_0'(\D) \ra K_0(\D)$ is injective.

\vspace{.2cm}
\n{\bf Exact sequences of derived categories.}
The main tool to compute $K_0(A_k)$ is the Thomason-Waldhausen Localization Theorem specialized to the case of DG categories.

A sequence of triangulated categories and exact functors
$\T_1 \stackrel{F_1}{\lra} \T_2 \stackrel{F_2}{\lra} \T_3 $
is called {\em exact} if $F_2F_1=0$, $F_1$ is fully faithful,
and $F_2$ induces an equivalence $\T_2/F_1(\T_1)\lra \T_3$,
see~\cite[Section 2.9]{Ke} and \cite[Section 3.1.5]{Sch2}.

A sequence of triangulated categories $\T_1 \stackrel{F_1}{\lra} \T_2 \stackrel{F_2}{\lra} \T_3$ is called {\em exact up to factors} if $F_2F_1=0$, $F_1$ is fully faithful, and $F_2$ induces an equivalence $\T_2/F_1(\T_1)\lra \T_3$ up to factors, see \cite[Definition 3.1.10]{Sch2}.
An inclusion $F:\mathcal{A}\lra \B$ of triangulated categories
is called an {\em equivalence up to factors} \cite[Definition 2.4.1]{Sch2} if every object of $\B$ is a direct summand of an object in $F(\mathcal{A})$.

Given a sequence of DG categories $\cal{A} \ra \cal{B} \ra \D$, if the sequence
$$D(\cal{A}) \ra D(\cal{B}) \ra D(\D)$$ of derived categories of DG modules is exact, then the associated sequence
$$D^c(\cal{A}) \ra D^c(\cal{B}) \ra D^c(\D)$$
of derived categories of compact objects is exact up to factors by Neeman's result \cite[Theorem 2.1]{Nee}.
According to \cite[Theorem 3.2.27]{Sch2}, the Thomason-Waldhausen Localization Theorem implies that there is an exact sequence of K-groups:
$$K_1(\D) \ra K_0(\cal{A}) \ra K_0(\cal{B}) \ra K_0(\D).$$

We will fit $D(A_k)$ into an exact sequence of derived categories and then make use of the localization theorem.

\begin{defn} \label{def rel proj}
For a DG algebra $A$, a DG $A$-module $Q$ is {\em (finitely generated) relatively projective} if it is a direct summand of a (finite) direct sum of modules of the form $A[n]$.
\end{defn}

We refer the reader to \cite[Section 3.1]{Ke} for the definition of {\em Property (P)} for a DG module.
Any relatively projective module has Property (P).
For any object $M \in D(A)$ there exists $P(M) \in D(A)$ which is isomorphic to $M$ in $D(A)$ and has Property (P) \cite[Theorem 3.1]{Ke}.
The object $P(M) \in D(A)$ is unique up to isomorphism.
If $A$ is an ordinary algebra viewed as a DG algebra concentrated in degree $0$, then $P(M)$ is a projective resolution of $M$.

Let $A,B$ be DG algebras, and $X$ be a DG left $A$, right $B$ bimodule.
We call $X$ a DG $(A, B)$-bimodule.
The derived tensor product functor $-\ot^{\mf{L}}_{A} X: D(A) \ra D(B)$ is defined by $M \ot^{\mf{L}}_{A} X \cong P(M) \ot_{A} X$.
Note that the derived tensor product commutes with infinite direct sums, see \cite[Section 6.1]{Ke}.

Consider $\alpha_k: A_{k-1} \ra A_k$ in (\ref{def alpha}).
Let $e_k=\alpha_k(1_{k-1})=(z\wtz) \ot 1_{k-1}$ which is an idempotent of $A_k$.
It generates a right ideal $I_k=e_kA_k$ of $A_k$.
The map $\alpha_k$ makes $I_k$ a DG $(A_{k-1}, A_k)$-bimodule and induces a functor
$$i_k=-\ot^{\mf{L}}_{A_{k-1}} I_k: D(A_{k-1}) \ra D(A_k).$$
Note that $I_k$ is relatively projective as a right $A_k$-module.

It is clear that $h_k \circ i_k$ is isomorphic to $\imath_k \circ h_{k-1}$ as functors $D(A_{k-1}) \ra D(\C_k)$, see (\ref{hkc}).
So
\begin{align}\label{iso K0 Ak}
K_0(\C) \cong \varinjlim K_0(A_k)
\end{align}
with respect to ${i_k}_*$.

The quotient map $A_k \ra L_k$ makes $\lk$ a DG $(A_k, \lk)$-bimodule, and induces a functor
$$j_k=-\ot^{\mf{L}}_{A_k} \lk: D(A_k) \ra D(L_k).$$

\vspace{.2cm}
\n{\bf A construction of $P(\lk)$ for $\lk \in D(A_k)$.}

For $k=1$, let $P(L)$ be a complex
$$\bigoplus\limits_{j \in \N} e_{jj}A_1 \ra A_1,$$
whose the differential is the sum of inclusions $\imath_j: e_{jj}A_1 \hookrightarrow A_1$, where $A_1$ is in degree $0$, $\bigoplus\limits_{j \in \N} e_{jj}A_1$ is in degree $-1$, and $e_{jj} \in J \subset A_1$ are idempotents.
In other words, $P(L)$ is the DG $A_1$-module
$$\left((\bigoplus\limits_{j \in \N} e_{jj}A_1[1]) \oplus A_1, \quad \bdry=\sum\limits_{j \in \N}\imath_j\right).$$
Since $\bigoplus\limits_{j \in \N} e_{jj}A_1=J$ as $A_1$-modules, $P(L) =(J \ra A_1) \cong L \in D(A_1)$.

For $k>1$, we take a product of $k$ copies of $P(L)$, where the product corresponds to the monoidal structure on $\C$.
More precisely, let $u(t,i)$ denote the idempotent of $A_k$ whose diagram consists of $k-1$ vertical long arcs and one pair of short arcs $e_{ii}$ as the $t$-th strand from the left, for $1 \leq t \leq k, i\in \N$, see Figure \ref{sec3fig7}.
They satisfy the commuting relations $u(t,i)u(t',i')=u(t',i')u(t,i)$ for $t \neq t'$.
So their products are also idempotents of $A_k$, denoted by $u(T,\mf{i})$ for $T \subset \{1,\dots,k\}$ and $\mf{i} \in \N^{|T|}$.
Here $u(\es,\es)$ is understood as the identity $1_k$ of $A_k$.
Let $$P(T,\mf{i}) =u(T,\mf{i})A_k$$
which is a relatively projective DG $A_k$-module.
For $P(T,\mf{i}), P(S,\mf{j})$ such that $T=S \sqcup \{r\}$ and $i_s=j_s$ for $s \in S$, there is an inclusion $\imath(T,\mf{i},r): P(T,\mf{i}) \ra P(S,\mf{j})$ of $A_k$-modules given by $\imath(T,\mf{i},r)(u(T,\mf{i}))=u(S,\mf{j})u(r,i_r)$.

\begin{figure}[h]
    \includegraphics[height=4cm]{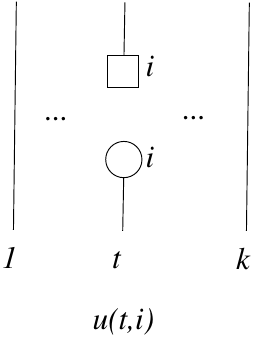}
    \caption{The idempotent $u(t,i) \in A_k$.}
    \label{sec3fig7}
  \end{figure}

Consider a DG $A_k$-module $P(\lk)$ given by a complex of relatively projective DG $A_k$-modules of finite length:
$$\bigoplus\limits_{|T|=k, \mf{i}\in\N^{|T|}} P(T,\mf{i}) \ra \bigoplus\limits_{|T|=k-1, \mf{i}\in\N^{|T|}} P(T,\mf{i}) \ra \cdots \ra \bigoplus\limits_{|T|=1, \mf{i}\in\N^{|T|}} P(T,\mf{i}) \ra A_k$$
with the differential
$$\bdry=\sum\limits_{T,\mf{i},r\in T} (-1)^{c(T,r)}\imath(T,\mf{i},r),$$
where $A_k$ is in degree $0$, and $c(T,r)=\#\{t\in T ~|~ t<r\}$.
The complex $P(\lk)$ is exact except at $A_k$.
Let
\begin{gather} \label{def pr}
\op{pr}: P(\lk) \ra \lk
\end{gather}
be the quotient map $A_k \ra \lk$ on the summand $A_k$, and the zero map on the remaining summands of $P(\lk)$.
Then $\op{pr}$ is an isomorphism in $D(A_k)$.

Except for the last term $A_k$, each $P(T,\mf{i})$ is naturally a submodule of the ideal $J_k$ of $A_k$, which is the kernel of the quotient map $A_k \ra L_k$, see Definition \ref{def Jk}.
This implies
\begin{gather} \label{iso lkotlk}
j_k(\lk)=\lk \ot^{\mf{L}}_{A_k} \lk \cong P(\lk)\ot_{A_k} \lk \cong A_k \ot_{A_k} \lk \cong \lk \in D(A_k).
\end{gather}

\begin{lemma} \label{lem exact}
The sequence of derived categories $D(A_{k-1}) \xra{i_k} D(A_k) \xra{j_k} D(\lk)$ is exact.
\end{lemma}
\begin{proof}
The image $i_k(A_{k-1})$ is isomorphic to the module $e_kA_k$ which is finitely generated relatively projective.
In particular, $i_k(A_{k-1})$ is compact in $D(A_k)$, and
$$\op{END}_{D(A_{k-1})}(A_{k-1}) \cong A_{k-1} \cong e_k A_{k}e_k \cong \op{END}_{D(A_{k})}(i_k(A_{k-1})).$$
The functor $i_k: D(A_{k-1}) \ra D(A_k)$ is fully faithful by \cite[Lemma 4.2]{Ke1}.

The composition $j_k \circ i_k$ sends the free module $A_{k-1}$ to
$$j_k \circ i_k (A_{k-1}) = j_k(A_{k-1} \ot_{A_{k-1}}I_k)\cong j_k(I_k) = e_kA_k \ot^{\mf{L}}_{A_k} L_k \cong e_kA_k \ot_{A_k} L_k=0,$$
where the last isomorphism holds since $e_kA_k$ is relatively projective, and the last equality holds since $e_kA_k$ is contained in the ideal $J_k$ which is the kernel of the quotient map $A_k \ra L_k$.
The composition $j_k \circ i_k$ commutes with the infinite direct sums \cite[Section 6.1]{Ke}.
Thus, $j_k \circ i_k=0$ on the smallest full triangulated subcategory of $D(A_{k-1})$ containing the free module $A_{k-1}$ and closed under infinite direct sums.
This full subcategory coincides with $D(A_{k-1})$, see \cite[Section 4.2]{Ke}.
It follows that $j_k \circ i_k=0$ on $D(A_{k-1})$.

The algebras $\lk$ and $A_k$ act on $\lk$ both from left and right via the map $A_k \ra L_k$.
In the following computation, we view $\lk$ in one of the three ways:
(1) as a right $\lk$-module, denoted $\lk^L$;
(2) as a right $A_k$-module, denoted $\lk^A$;
and (3) as a $(A_k, \lk)$-bimodule, denoted $^A\lk^L$.

The functor $j_k$ admits a right adjoint functor $f_k: D(\lk) \ra D(A_k)$ which is the restriction functor with respect to the quotient map $A_k \ra \lk$.
In particular, $f_k(\lk^L)=\lk^A$.
The functor $f_k$ is fully faithful if and only if the counit map
\begin{gather} \label{iso counit}
\delta_{\lk}: \lk^{A} \otimes^{\mf{L}}_{A_k} {^{A}} \lk ^{L} \ra \lk ^{L}
\end{gather}
is an isomorphism of right $\lk$-modules, see \cite[Lemma 4 (1,3)]{NS}.
The counit map $\delta_{\lk}$ is the image of $1_{\lk} \in \Hom_{D(A_k)}(\lk^A,\lk^A)=\Hom_{D(A_k)}(\lk^A,f_k(\lk^L))$ under the adjunction isomorphism $ad: \Hom_{D(A_k)}(\lk^A,f_k(\lk^L)) \cong \Hom_{D(\lk)}(\lk^A \otimes^{\mf{L}}_{A_k} {^{A}} \lk ^{L}, \lk ^{L})$.
Replacing $\lk^A$ by its resolution $P(\lk)$, there is a chain of isomorphisms
\begin{align*}
\Hom_{D(A_k)}(\lk^A,f_k(\lk^L)) & \xra{f} \Hom_{D(A_k)}(P(\lk),f_k(\lk^L)) \xra{ad} \Hom_{D(\lk)}(P(\lk) \otimes_{A_k} {^{A}} \lk ^{L} , \lk^L) \\
& \xra{g} \Hom_{D(\lk)}(\lk^L,\lk^L) \xra{h} \Hom_{D(\lk)}(\lk^{A} \otimes^{\mf{L}}_{A_k} {^{A}} \lk ^{L}, \lk ^{L}).
\end{align*}
Here $g$ and $h$ are induced by (\ref{iso lkotlk}).
Recall $\op{pr}: P(\lk) \ra \lk$ from (\ref{def pr}), and let $m: \lk \otimes_{A_k} \lk \ra \lk$ denote the multiplication map.
Then
$$\delta_{\lk}=h\circ g\circ ad \circ f(1_{\lk})=h\circ g \circ ad(\op{pr})=h\circ g(m \circ (\op{pr} \ot 1_{\lk}))=h(1_{\lk})$$
which is an isomorphism.
It follows that $f_k: D(\lk) \ra D(A_k)$ is fully faithful.

Let $\cal{T}_{k}=D(A_k) / i_k(D(A_{k-1}))$, and $q_k: D(A_k) \ra \cal{T}_k$ denote the quotient functor.
Since $j_k \circ i_k$ is zero, the functor $j_k$ factors through $t_k: \cal{T}_k\ra D(\lk)$.
Let $$s_k=q_k\circ f_k: D(\lk) \ra D(A_k) \ra \cal{T}_k.$$
It is clear that $t_k \circ s_k=t_k \circ q_k\circ f_k=j_k \circ f_k$ is an equivalence since the counit map $\delta_{\lk}: j_k \circ f_k(\lk) \ra \lk$ in (\ref{iso counit}) is an isomorphism, and the conditions of Lemma 4.2(a,c) in \cite{Ke1} hold.

It remains to show that $s_k$ is an equivalence.
By \cite[Lemma 4.2 (a,c)]{Ke1}, it is enough to show that $s_k(\lk)$ is a compact generator of $\cal{T}_k$, and ${s_k}_*: \Hom_{D(\lk)}(\lk,\lk[n]) \ra \Hom_{\cal{T}_k}(s_k(\lk),s_k(\lk)[n])$ is an isomorphism.
The object $$s_k(\lk)=q_k\circ f_k(\lk) \cong q_k(P(\lk)) \cong q_k(A_k),$$ since all other terms except for $A_k$ in $P(\lk)$ lie in $i_k(D(A_{k-1}))$.
Theorem 2.1 in \cite{Nee} implies that $q_k(A_k)$ is a compact object of $\cal{T}_{k}$ since $A_k$ is a compact object of $D(A_k)$.
Moreover, $\{q_k(A_k)\}$ generates $\cal{T}_{k}$ since $\{A_k\}$ generates $D(A_k)$.
We have ${s_k}_*={q_k}_* \circ {f_k}_*$, where
$${f_k}_*: \Hom_{D(\lk)}(\lk,\lk[n]) \ra \Hom_{D(A_k)}(f_k(\lk),f_k(\lk)[n]),$$
$${q_k}_*: \Hom_{D(A_k)}(f_k(\lk),f_k(\lk)[n]) \ra \Hom_{\cal{T}_k}(s_k(\lk),s_k(\lk)[n]).$$
The map ${f_k}_*$ is an isomorphism since $f_k$ is fully faithful.
The map ${q_k}_*$ is an isomorphism if $\Hom_{D(A_k)}(i_k(M),f_k(\lk)[n])=0$ for any $M \in D(A_{k-1})$ by \cite[Definition 9.1.3, Lemma 9.1.5]{Nee1}.
By adjointness $\Hom_{D(A_k)}(i_k(M),f_k(\lk)[n])=\Hom_{D(\lk)}(j_k \circ i_k(M), \lk[n])=0$ since $j_k \circ i_k=0$.
We finally conclude that $s_k$ is an equivalence.
\end{proof}

There is an exact sequence of K-groups
\begin{gather} \label{les}
K_1(\lk) \xra{\bdry} K_0(A_{k-1}) \xra{{i_k}_*} K_0(A_k) \xra{{j_k}_*} K_0(\lk),
\end{gather}
induced by the exact sequence of the derived categories in Lemma \ref{lem exact}.

To compute $K_0(A_k)$ we need $K_i(\lk)$ for $i=0,1$.


\subsection{K-theory of $\lk$}
We compute $K_i(\lk)$ for $i=0,1$ in this subsection.
The key tool to compute $K_0(\lk)$ is a result of Keller \cite[Theorem 3.1(c)]{Ke1}.
Recall the notion of relatively projective DG modules from Definition \ref{def rel proj}.
For any DG right $\lk$-module $M$, let $M[1]$ denote the shift of $M$, where $M[1]^i=M^{i+1}$, $d_{M[1]}=-d_M$, and $m[1]\cdot a=ma [1]$ for $m \in M, m[1] \in M[1]$ and $a \in \lk$.
See Sections 10.3 and 10.6.3 in \cite{BL} for definitions of shifts of left and right DG modules, respectively.

\begin{thm}[Keller \cite{Ke1}] \label{thm Keller}
Given any DG algebra $A$ and  a DG $A$-module $M$, let
$$\cdots \ra \overline{Q_n} \ra \cdots \ra \overline{Q_0} \ra H^*(M) \ra 0$$
be a projective resolution of $H^*(M)$ viewed as graded $H^*(A)$-module such that $\overline{Q_n} \xra{\sim} H^*(Q_n)$ for a relatively projective DG $A$-module $Q_n$. Then $M$ is isomorphic to a module $P(M)$ in the derived category $D(A)$ which admits a filtration $F_n$ such that $\bigcup\limits_{n=0}^{\infty} F_n=P(M)$, the inclusion $F_{n-1} \subset F_n$ splits as an inclusion of graded $A$-modules, and
$F_n / F_{n-1} \xra{\sim} Q_n[n]$ as DG $A$-modules.
\end{thm}

We specialize to the case $A=\lk$.
There is an isomorphism of free right DG $\lk$-modules
\begin{gather} \label{iso lk}
h_k: \lk \simeq \lk[1]
\end{gather}
given by $h_k(m)=a_k\cdot m$ for $m \in \lk$, where the multiplication is that of the algebra $\lk$ and $a_k$ is the invertible closed element of degree $1$.
So $Q$ is relatively projective if it is a direct summand of a free module $\lk^I$, where $I$ is the index set.
Since the differential is trivial on $\lk$, $H^*(\lk) \cong \lk$ as graded algebras.
So $Q$ is a relatively projective DG $\lk$-module if and only if $Q \cong H^*(Q)$ is a direct summand of $\lk^I$ as graded $\lk$-module.
Given any projective resolution of $H^*(M)$ as in Theorem \ref{thm Keller} we can take $Q_n=\overline{Q_n}$ viewed as a DG $\lk$-module.

We now consider projective resolutions of $N=H^*(M)$.
Let $N=\bigoplus N^i$ be its decomposition into homogenous components.
Let $\rk$ denote the degree zero subalgebra of the graded algebra $\lk$.
Then $\rk$ is generated by $b_i=a_i a_k^{-1}$ for $1 \leq i \leq k-1$, and
\begin{gather} \label{def rk}
\rk=\K \lan b_1^{\pm1}, ..., b_{k-1}^{\pm1} \ran / (b_ib_j=-b_jb_i, i \neq j).
\end{gather}
We fix the inclusion $\rk \ra \lk$ from now on.
There is an isomorphism of graded algebras
\begin{gather} \label{iso lkrk}
\lk \cong \rk\lan a_k^{\pm1} \ran / (b_ia_k=-a_kb_i).
\end{gather}
For any graded $\lk$-module $N$, each component $N^i$ is a $\rk$-module.
Since $a_k$ is invertible of degree $1$, any graded $\lk$-module $N$ is completely determined by $N^0$ as a $\rk$-module.
More precisely, the action of $a_k$ induces an isomorphism of $\rk$-modules
\begin{align} \label{iso twist}
N^{i+1} \ra \alpha(N^{i}),
\end{align}
where $\alpha(N^{i})$ is the abelian group $N^i$ with the $\alpha$-twisted action of $\rk$ via an automorphism $\alpha:\rk \ra \rk$ given by $\alpha(b_i)=-b_i$.
Any projective resolution $P(N^0)$ of a $\rk$-module $N^0$ induces a projective resolution $P(N^i)$ of $N^i$.
The direct sum $\bigoplus\limits_{i\in\Z}P(N^i)$ is a projective resolution of the graded $\lk$-module $N$.

We recall the following results about $\rk$ studied by Farrell and Hsiang \cite{FH}.
The algebra $\rk \cong R_{k-2}[b_{k-1}^{\pm 1}]$ is an $\alpha$-twisted finite Laurent series ring.
According to \cite[Theorem 25]{FH}, $\rk$ is right regular.
So any finitely generated $\rk$-module admits a finite resolution by finitely generated projective $\rk$-modules.
Furthermore, $K_0(\rk) \cong \Z$ with a generator $[\rk]$ by \cite[Theorem 27]{FH}.

Any isomorphism class of objects in $D^c(\lk)$ has a representative $M$ which is isomorphic to a direct summand of $\lk^{\oplus r}$ for some finite $r$ as graded $\lk$-modules (ignoring the differential).
So $M^0$ and $H^0(M)$ are finitely generated $\rk$-modules since $\rk$ is Noetherian.
Then $H^0(M)$  admits a finite resolution by finitely generated projective $\rk$-modules.
The graded $\lk$-module $H^*(M)$ admits a finite resolution by finitely generated projective $\lk$-modules.
We have the following lemma by applying Keller's Theorem \ref{thm Keller}.

\begin{lemma} \label{lem p resolution}
Any $M$ in $D^c(\lk)$ is isomorphic to $P(M)$ which admits a finite filtration $F_n(M)$ such that $F_n(M) / F_{n-1}(M) \xra{\sim} Q_n(M)[n]$ is a finitely generated relatively projective DG $\lk$-module.
\end{lemma}

\begin{lemma} \label{lem K0 of L}
There is a surjection of abelian groups $\eta_k: \Z/2 \ra K_0(\lk)$.
\end{lemma}
\begin{proof}
By Lemma \ref{lem p resolution} we have
$$[M]=\sum\limits_n (-1)^n[Q_n(M)] \in K_0(\lk)$$
for $M$ in $D^c(\lk)$, where the sum is a finite sum.
The abelian group $K_0(\lk)$ is generated by classes $[Q]$ of finitely generated relative projective $Q$.

The inclusion $\rk \ra \lk$ is a map of unital DG algebras, where $\rk$ is viewed as a DG algebra concentrated in degree $0$.
It induces a functor $g_k: D^c(\rk) \ra D^c(\lk)$ given by tensoring with the $(\rk,\lk)$-bimodule $\lk$.
Any $Q$ is a direct summand of a finite free module $\bigoplus \lk$, and has the trivial differential.
Its degree zero component $Q^0$ is a finitely generated projective $\rk$-module, and the action of $a_k$ induces an isomorphism of $\rk$-modules $Q^{i+1} \ra \alpha(Q^{i})$.
We have
$$g_k(Q^0) = Q^0 \otimes_{\rk} \lk =\bigoplus\limits_{i \in \Z}Q^0 \otimes \rk a_k^i \cong \bigoplus\limits_{i \in \Z}Q^i = Q,$$
by (\ref{iso lkrk}), where the direct sums are taking as $\rk$-modules.
It follows that ${g_k}_*: K_0(\rk) \ra K_0(\lk)$ is surjective since ${g_k}_*([Q^0])=[Q]$.
The group $K_0(\lk)$ is generated by $[\lk]={g_k}_*([\rk])$ since $K_0(\rk) \cong \Z$ with a generator $[\rk]$, see  \cite[Theorem 27]{FH}.
Isomorphism (\ref{iso lk}) implies that $[\lk]=-[\lk]$.
Hence the map $\eta_k: \Z/2 \ra K_0(\lk)$ defined by $\eta_k(1)=[\lk]$ is surjective.
\end{proof}

According to \cite[Section 3.2.12]{Sch2}, the K-space $K(\E,w)$ of a complicial exact category $\E$ with weak equivalences $w$ is the homotopy fiber of of $BQ(\E^w) \ra BQ(\E)$, where $\E^w \subset \E$ is the full exact subcategory of objects $X$ in $\E$ for which the map $0 \ra X$ is a weak equivalence.
Here $BQ(\E)$ is the classifying space of the category $Q(\E)$ used in Quillen's Q-construction.
By definition, there is a an exact sequence:
\begin{gather} \label{eq k1}
K_1(\E^w) \ra K_1(\E) \ra K_1(\E,w) \ra K_0(\E^w) \xra{i} K_0(\E) \ra K_0(\E,w).
\end{gather}
Here, $K_i(\E^w)$ and $K_i(\E)$ are K groups of the exact categories $\E^w$ and $\E$, respectively.

From now on let $\E$ denote the complicial exact category $\B(\lk)^{pre}$, see Definition \ref{def B(A)}.
A sequence $L \ra M \ra N$ is exact if it is split exact when forgetting the differential.
The weak equivalences are the homotopy equivalences.
So $\E^w \subset \E$ is the full subcategory of contractible objects in $\E$.
By \cite[Proposition 3.2.22]{Sch2} and the definition of $K_1$ in (\ref{notation1}) and (\ref{notation2})
\begin{gather} \label{eq k0k1}
K_0(\E,w) \cong K_0'(\lk), \qquad K_1(\E,w)=K_1(\lk).
\end{gather}

Any $M \in \E$ is a finite direct sum of free modules $\lk[n]$ when forgetting the differential.
Since $\lk \cong \lk[1] \in \E$, any $M$ is isomorphic to $\lk^{\oplus r}$ for some $r \in \N$ as graded $\lk$-modules.
Its degree zero component $M^0 \cong \rk^{\oplus r}$ as free $\rk$-modules.
Since any exact sequence $L \ra M \ra N$ in $\E$ induces a split exact sequence $L^0 \ra M^0 \ra N^0$ of free $\rk$-modules, it induces a homomorphism $r: K_0(\E) \ra K_0(\rk) \cong \Z$ defined by $r([M])=[M^0] \in K_0(\rk)$.
It is clear that $r$ is surjective.

\begin{lemma} \label{lem k0E}
The group $K_0(\E) \cong \Z$ with a generator $[\lk]$.
\end{lemma}
\begin{proof}
Define a homomorphism $\Z[q,q^{-1}] \ra K_0(\E)$ by mapping $q^{n}$ to the class $[\lk[n]]$ for $n\in\Z$.
It is surjective since any object $M \in \E$ admits a finite filtration whose subquotients are finite direct sums of free modules $\lk[n]$.
The map factors through $\Z[q,q^{-1}] \xra{q=1} \Z$ since $\lk \cong \lk[1] \in \E$.
Let $\phi: \Z \ra K_0(\E)$ denote the induced map which is surjective.
It is clear that $r$ and $\phi$ are inverse to each other.
\end{proof}

We now consider $K_0(\E^w)$.
Any object $M \in \E^w$ is contractible.
There exists a degree $-1$ map $h: M \ra M$ of graded $\lk$-modules such that $dh+hd=1$ on $M$.
Thus, $\Ker d_M = \IM d_M$.
Moreover, $\Ker d_M$ and $h(\IM d_M)$ are graded $\lk$-submodules of $M$.
For any $m \in \IM d_M \cap h(\IM d_M)$, $m=h(n)$ for some $n \in \IM d_M$ so that $n=dh(n)+hd(n)=d(m)=0$ and $m=0$.
Thus $\IM d_M \cap h(\IM d_M)=\{0\}$.
As graded $\lk$-modules, $M \cong (\IM d_M \oplus h(\IM d_M))$ since $m=dh(m)+hd(m)$.
Moreover, $\IM d_M$ is a DG $\lk$-submodule of $M$ with the trivial differential.
As DG $\lk$-modules,
\begin{align} \label{iso dg}
\begin{split}
M & \cong (h(\IM d_M) \oplus \IM d_M, d=d_M:h(\IM d_M) \ra \IM d_M) \\
& \cong  (\IM d_M[1] \oplus \IM d_M, d=id: \IM d_M[1] \ra \IM d_M).
\end{split}
\end{align}

The degree zero component $(\IM d_M)^0$ is a direct summand of a finitely generated free $\rk$-module $M^0$.
Thus $(\IM d_M)^0$ is a finitely generated projective $\rk$-module.
Recall from \cite[Theorem 27]{FH} that $K_0(\rk) \cong \Z$ with a generator of the class $[\rk]$ of the free module $\rk$.
It follows that every finitely generated projective $\rk$-module $P$ is stably free, i.e. $P\oplus \rk^m \cong \rk^n$ for some $m,n \in \N$.
Thus $(\IM d_M)^0$ is a finitely generated stably free $\rk$-module so that $\IM d_M$ is a stably free DG $\lk$-module.
Let $C(\lk)=Cone(\lk \xra{id} \lk) \in \E^w$, where two $\lk$'s are in degrees $-1$ and $0$.
There exists $m,n \in \N$ such that $M\oplus C(\lk)^{\oplus m} \cong C(\lk)^{\oplus n}$ by (\ref{iso dg}).

Define a homomorphism $\psi: \Z \ra K_0(\E^w)$ by $\psi(1)=[C(\lk)]$.
Then $\psi$ is surjective.

For $M \in \E^w$, let
\begin{gather} \label{def t}
t(M)=[(\IM d_M)^0] \in K_0(\rk) \cong \Z.
\end{gather}
For any exact sequence $L \xra{f} M \xra{g} N$ in $\E^w$, there is an induced sequence
$$\IM d_L \xra{f} \IM d_M \xra{g} \IM d_N.$$
We claim that it is a short exact sequence of graded $\lk$-modules.

\n(1) The first map is clearly injective.

\n(2) The last map is surjective. For any $n=d_N(n') \in \IM d_N$, $n'=g(m')$ for some $m' \in M$ since $g$ is surjective. So $n=d_N(g(m'))=g(d_M(m')) \in g(\IM d_M)$.

\n(3) The middle term is exact. For any $m \in \IM d_M \cap \Ker(g)$, $m=f(l)$ for some $l \in L$.  Then $d_M(m)=d_M(f(l))=f(d_L(l))=0$ implies that $d_{L}(l)=0$ since $f$ is injective.
So $l \in \Ker d_L=\IM d_L$ and $m=f(l) \in f(\IM d_L)$.

Then $(\IM d_L)^0 \xra{f} (\IM d_M)^0 \xra{g} (\IM d_N)^0$ is a short exact sequence of finitely generated stably free $\rk$-modules.
Therefore, the map $t$ given by (\ref{def t}) induces a homomorphism $t: K_0(\E^w) \ra K_0(\rk) \cong \Z$ which maps $[C(\lk)]$ to $1$.
It is clear that $t$ and $\psi$ are inverse to each other.
We have the following lemma.

\begin{lemma} \label{lem k0Ew}
The group $K_0(\E^w) \cong \Z$ with a generator $[C(\lk)]$.
\end{lemma}

The map $i: K_0(\E^w) \ra K_0(\E)$ in the exact sequence (\ref{eq k1}) takes the generator $[C(\lk)]$ to $2[\lk]$, see Lemma \ref{lem k0E}.
In particular, $i$ is injective and not surjective.

\begin{prop} \label{prop K0 of L}
The group $K_0(\lk)\cong\Z/2$ with a generator $[\lk]$.
\end{prop}
\begin{proof}
Recall from (\ref{notation1}) that $K_0'(\lk) \cong K_0(\op{Ho}(\C^{pre}))$, and $K_0'(\lk) \ra K_0(\lk)$ is injective by \cite[Corollary 2.3]{Th}.
The group $K_0'(\lk)$ is nonzero since $i$ is not surjective in the exact sequence (\ref{eq k1}).
It implies that $K_0(\lk)$ is nonzero.
Hence the surjection $\eta_k: \Z/2 \ra K_0(\lk)$ in Lemma \ref{lem K0 of L} is an isomorphism.
\end{proof}

Define a map $\rho_k: \Z \ra K_0(A_k)$ of abelian groups by $\rho_k(1)=[A_k]$.
\begin{prop} \label{prop surj Ak}
The map $\rho_k$ is surjective for all $k \geq 0$.
\end{prop}
\begin{proof}
There is a commutative diagram from the exact sequence (\ref{les}) and Proposition \ref{prop K0 of L}:
$$\xymatrix{
K_0(A_{k-1}) \ar[r]^{{i_k}_*} & K_0(A_k) \ar[r]^{{j_k}_*} & K_0(\lk)  &\\
\Z \ar[r]^{2} \ar[u]^{\rho_{k-1}} & \Z \ar[r] \ar[u]^{\rho_{k}} & \Z_2  \ar[u]^{\cong} \ar[r] &0.
}$$
The first square commutes because $[A_k]={i_k}_*[A_{k-1}]+[A_k[-1]]$ by an analogue of (\ref{eq iso2}).
The second square commutes since ${j_k}_*[A_k]=[L_k]$.
The map $\rho_0$ is an isomorphism since $A_0 \cong \K$ by (\ref{def A0}).
By induction on $k$ one can prove that $\rho_k$ is surjective using the snake lemma.
\end{proof}

Taking the direct limit of $\rho_k$, we have a surjective map $\rho: \varinjlim \Z \ra \varinjlim K_0(A_k)$ of abelian groups, where $\varinjlim \Z \cong \Z[\frac{1}{2}]$, and $\varinjlim K_0(A_k)$ with respect to ${i_k}_*$ is naturally isomorphic to $K_0(\C)$ as abelian groups, see (\ref{iso K0 Ak}).
The monoidal structure on $D^c(\C)$ induces a ring structure on $K_0(\C)$, where $[A_k][A_{k'}]=[A_{k+k'}] \in K_0(\C)$.
Then $\rho: \Z[\frac{1}{2}] \ra K_0(\C)$ is a ring homomorphism such that $\rho(1)=[A_0]$, and $\rho(\frac{1}{2})=[A_1]$.

\begin{thm} \label{thm surj K0C}
There is a surjective homomorphism of rings: $\rho: \Z[\frac{1}{2}] \ra K_0(\C)$.
\end{thm}

\vspace{.2cm}
To show that $\rho$ is an isomorphism, we need to assume that $\op{char}(\K)=2$.
The exact sequence (\ref{eq k1}) gives
$K_1(\E^w) \ra K_1(\E) \ra K_1(\lk) \ra 0,$ since $i$ is injective.

\begin{prop} \label{prop K1 of L}
If $\op{char}(\K)=2$, then $K_1(\lk)$ is $2$-torsion.
\end{prop}
\begin{proof}
It is enough to show that $2\alpha$ is in the image of $K_1(\E^w) \ra K_1(\E)$ for any $\alpha \in K_1(\E)$.
We use Nenashev's presentation of $K_1(\E)$ of the exact category $\E$, see \cite{Ne}.
Any $\alpha \in K_1(\E)$ is represented by a double short exact sequence
$$\xymatrix{
M \ar@<.5ex>[r]^{f_1} \ar@<-.5ex>[r]_{g_1}  & N \ar@<.5ex>[r]^{f_2} \ar@<-.5ex>[r]_{g_2} & L
}$$

Consider the cone $C(M)$ of $id: M \ra M$, and morphisms $i_M: M \ra C(M)$ and $j_M: C(M) \ra M[1]$ in $\E$.
Any morphism $f: M \ra N$ induces two morphisms $f[1]: M[1] \ra N[1]$ and $C(f): C(M) \ra C(N)$.
Let $\alpha[1], C(\alpha) \in K_1(\E)$ be the classes of double short exact sequences consisting of $M[1], N[1], L[1]$ and $C(M), C(N), C(L)$, respectively.
The following diagram
$$\xymatrix{
M \ar@<.5ex>[r]^{f_1} \ar@<-.5ex>[r]_{g_1} \ar@<.5ex>[d]^{i_M} \ar@<-.5ex>[d]_{i_M}  & N \ar@<.5ex>[r]^{f_2} \ar@<-.5ex>[r]_{g_2} \ar@<.5ex>[d]^{i_N} \ar@<-.5ex>[d]_{i_N} & L \ar@<.5ex>[d]^{i_L} \ar@<-.5ex>[d]_{i_L}\\
C(M) \ar@<.5ex>[r]^{C(f_1)} \ar@<-.5ex>[r]_{C(g_1)} \ar@<.5ex>[d]^{j_M} \ar@<-.5ex>[d]_{j_M} & C(N) \ar@<.5ex>[r]^{C(f_2)} \ar@<-.5ex>[r]_{C(g_2)} \ar@<.5ex>[d]^{j_N} \ar@<-.5ex>[d]_{j_N} & C(L) \ar@<.5ex>[d]^{j_L} \ar@<-.5ex>[d]_{j_L} \\
M[1] \ar@<.5ex>[r]^{f_1[1]} \ar@<-.5ex>[r]_{g_1[1]}  & N[1] \ar@<.5ex>[r]^{f_2[1]} \ar@<-.5ex>[r]_{g_2[1]} & L[1]
}$$
satisfies Nenashev's condition in \cite[Proposition 5.1]{Ne}.
Hence $$\alpha - C(\alpha) + \alpha[1]=\beta_M - \beta_N + \beta_L \in K_1(\E),$$
where $\beta_X$ is the class of the vertical double short exact sequence consisting of $X, C(X), X[1]$ for $X=M,N,L$.
By \cite[Lemma 3.1]{Ne}, $\beta_X=0 \in K_1(\E)$.
So $\alpha + \alpha[1] = C(\alpha)$, which is in the image of $K_1(\E^w) \ra K_1(\E)$.

If $\op{char}(\K)=2$, then $\lk=\K\lan a_1^{\pm1}, \dots, a_k^{\pm1} \ran / (a_ia_j=a_ja_i, i \neq j)$ is commutative as in Proposition \ref{prop iso}.
In particular, $a_i$ is central, invertible, closed and of degree $1$.
Define $h_M: M \ra M[1]$ by $h_M(m)=ma_1$ which is an endomorphism of $M$ in $\E$.
It is an isomorphism since $a_1$ is invertible.
Moreover, $h_N \circ f=f[1] \circ h_M$ for any morphism $f: M \ra N$ of $\E$.
The following diagram
$$\xymatrix{
M \ar@<.5ex>[r]^{f_1} \ar@<-.5ex>[r]_{g_1} \ar@<.5ex>[d]^{h_M} \ar@<-.5ex>[d]_{h_M}  & N \ar@<.5ex>[r]^{f_2} \ar@<-.5ex>[r]_{g_2} \ar@<.5ex>[d]^{h_N} \ar@<-.5ex>[d]_{h_N} & L \ar@<.5ex>[d]^{h_L} \ar@<-.5ex>[d]_{h_L}\\
M[1] \ar@<.5ex>[r]^{f_1[1]} \ar@<-.5ex>[r]_{g_1[1]} \ar@<.5ex>[d]^{0} \ar@<-.5ex>[d]_{0} & N[1] \ar@<.5ex>[r]^{f_2[1]} \ar@<-.5ex>[r]_{g_2[1]} \ar@<.5ex>[d]^{0} \ar@<-.5ex>[d]_{0} & L[1] \ar@<.5ex>[d]^{0} \ar@<-.5ex>[d]_{0} \\
0 \ar@<.5ex>[r]^{0} \ar@<-.5ex>[r]_{0}  & 0 \ar@<.5ex>[r]^{0} \ar@<-.5ex>[r]_{0} & 0
}$$
satisfies Nenashev's condition in \cite[Proposition 5.1]{Ne}.
It implies that $\alpha-\alpha[1]+0=0-0+0$ by \cite[Lemma 3.1]{Ne}.
Hence $\alpha=\alpha[1]$, and $2\alpha$ is in the image of $K_1(\E^w) \ra K_1(\E)$ for any $\alpha \in K_1(\E)$.
\end{proof}

\begin{rmk} \label{rmk char}
For the field $\K$ of any characteristic, there is a central element $a_1\cdots a_k$ of degree $k$ if $k$ is odd.
So $\alpha=\alpha[k]$, and $K_1(L_k)$ is $2$-torsion.
But the same argument does not apply if $k$ is even.
\end{rmk}

\begin{rmk}
It can be computed from \cite{Sch1} that $K_0(L_1) \cong \Z/2[L_1]$ and $K_1(L_1) \cong \K^*/(\K^*)^2$ which is $2$-torsion for the ground field $\K$ of any characteristic.
\end{rmk}

\begin{thm} \label{thm main}
There is an isomorphism of rings $K_0(\C) \cong \Z[\frac{1}{2}]$ if $\op{char}(\K)=2$.
\end{thm}
\begin{proof}
We want to show that $\rho_k: \Z \ra K_0(A_k)$ is an isomorphism by induction on $k$.
It is true when $k=0$ since $A_0 \cong \K$ by (\ref{def A0}).
Suppose that $\rho_{k-1}$ is an isomorphism.
If $\op{char}(\K)=2$, then $K_1(\lk)$ is $2$-torsion so that the map $\bdry: K_1(\lk) \ra K_0(A_{k-1}) \cong \Z$ is zero.
The commutative diagram in the proof of Proposition \ref{prop surj Ak} becomes
$$\xymatrix{
K_1(\lk) \ar[r]^{\bdry=0} & K_0(A_{k-1}) \ar[r]^{{i_k}_*} & K_0(A_k) \ar[r]^{{j_k}_*} & K_0(\lk)  &\\
& \Z \ar[r]^{2} \ar[u]^{\rho_{k-1}} & \Z \ar[r] \ar[u]^{\rho_{k}} & \Z_2  \ar[u]^{\cong} \ar[r] &0.
}$$
The snake lemma implies that $\rho_k$ is an isomorphism, and $K_0(A_k)$ is freely generated by $[A_k]$.

The same argument before Theorem \ref{thm surj K0C} shows that $K_0(\C) \cong \Z[\frac{1}{2}]$ as rings.
\end{proof}

\begin{rmk}\label{remark-change-sign}
If degrees of the generators $x$ and $y$ in Figure~\ref{sec3fig1} are exchanged, i.e. $\op{deg}(x)=-1, \op{deg}(y)=1$, then the resulting derived category will have an isomorphism $X^\we \cong \mb^\we \oplus X^\we[1]$ as the analogue of equation (\ref{eq iso}).
Since the quotient algebra $\lk$ is unchanged under the exchange of the degrees, Theorem~\ref{thm main} also holds in this case.

If $\op{deg}(x)=m, \op{deg}(y)=-m$ for some odd $m \neq \pm1$, then in the resulting derived category there is an isomorphism $X^\we \cong \mb^\we \oplus X^\we[-m]$, leading to a homomorphism from $\Z[\frac{1}{2}]$ to the Grothendieck ring of that category.
We don't know whether it is an isomorphism.
\end{rmk}


\subsection{A p-DG extension}

Witten-Reshetikhin-Turaev 3-manifold invariants, when extended to a 3-dimensional TQFT, require working over the ring $\ZNxi\subset \CC$, where $\xi$ is a primitive $N$-th root of unity. The space associated to a (decorated) surface in the TQFT is a free module over $\ZNxi$ and the maps associated to cobordisms are $\ZNxi$-linear.
This ring contains the subring $\Z[\xi]$ of cyclotomic integers.

When $N$ is a prime $p$, the ring $\Z[\xi]\cong \Z[q]/(1+q+\dots +q^{p-1})$. In the
notation, $\xi=e^{\frac{2\pi i}{p}}$ is an element of $\CC$
while $q$ is a formal variable, and the isomorphism takes $q$ to $\xi$.
Let us also denote this ring by $R_p$.
Ring $R_p$ admits a categorification, investigated in~\cite{Kh1,Qi}.
One works over
a field $\K$ of characteristic $p$ and forms a graded  Hopf algebra $H=\K[\partial]/(\partial^p)$, with $\deg(\partial) = 1$. The category of finitely-generated graded $H$-modules has a quotient category, called the stable category,
where morphisms which factor through a projective module are set to $0$.
The stable category $H\umod$ is triangulated monoidal and its
Grothendieck ring $K_0(H\umod)$ is naturally isomorphic to the cyclotomic
ring $R_p$. Multiplication
by $q$ corresponds to the grading shift $\{1\}$ in the category of graded
$H$-modules. The shift functor $[1]$ in the triangulated category $H\umod$
is different from the grading shift functor $\{ 1\}$.

We now explain a conjectural way to enhance this categorification of
$R_p$
using a version of isomorphism (\ref{eq1}) from the introduction to
categorify the ring $\Zpxi$ which contains both $R_p$ and $\Z[\xi]$ as subrings. The point is that in $\Z[\xi]$
there is an equality of principal ideals
$$(p) = (1-\xi)^{p-1} \ $$
see~\cite[Proposition 6.2]{Mi}, so that subrings $\Zpxi$ and $\Z\big[\frac{1}{1-\xi},\xi\big]$
of $\C$ coincide (equivalently, localizations $R_p\big[p^{-1}\big]$
and $R_p\big[(1-q)^{-1}\big]$ are isomorphic). Inverting $p$ is equivalent to inverting $1-\xi$,
and the latter can be inverted using a variation of isomorphism (\ref{eq1}).

Namely, one would like to have
a monoidal category  over a field $\K$ of characteristic
$p$ with a generating object $X$ and an isomorphism
\begin{equation} \label{eqPdg}
  X \cong X\{ 1\} \oplus \mb,
  \end{equation}
where $\{1\}$ is degree shift by one. Having this isomorphism requires the
four generating morphisms as in Figure~\ref{sec3fig5}, denoted $x,y,z,z^{\ast}$
 in Figure~\ref{sec3fig1}.
The degrees are now the opposite, $\deg(x)=-1,$ $\deg(y)=1$,
$\deg(z)=\deg(z^{\ast})=0$. The defining relations are the same, see Figure~\ref{sec3fig2}, but now all far-away generating morphisms commute rather than
super-commute.

The construction gives rise to a graded pre-additive category $\C$ with
objects -- tensor powers of $X$ and morphisms being planar diagrams built out
of generators subject to defining relations. We make $\C$ into a $p$-DG
category by equipping it with the derivation $\partial$ of degree $1$
that acts by zero on all generating morphisms, hence on all morphisms.

We then extend category $\C$ to a triangulated
category, as explained in Section~\ref{subsec-DG-ext}
for the DG case, by substituting the $p$-DG version everywhere.
We pass to the pre-triangulated $p$-DG category $\C^{pre}$ by formally
adding iterated tensor products with objects of $H\umod$,
finite direct sums and cones of morphisms. Shifts of objects are included
in this construction, since they are isomorphic to tensor products
with one-dimensional graded $H$-modules. The homotopy category
$\mathrm{Ho}(\C^{pre})$ is triangulated, and we define $\wt{\C}$ to
be its idempotent completion. The category $\wt{\C}$ is triangulated
monoidal Karoubi closed, and there is a natural ring homomorphism
\begin{equation}\label{p-homomorphism-eq}
  \Z\biggl[ \frac{1}{p},\xi\biggr] \ \lra \ K_0(\wt{\C})
  \end{equation}
taking $(1-\xi)^{-1}$ to $[X]$.

\vspace{0.05in}

\begin{prob} Is the map (\ref{p-homomorphism-eq}) an isomorphism?
\end{prob}

Beyond this problem, there is an open question whether
category $\wt{\C}$ can be used to enhance known categorifications of
quantum groups at prime roots of unity and to help with categorification
of the Witten-Reshetikhin-Turaev 3-manifold invariants at prime roots.


\section{A monoidal envelope of Leavitt path algebras}
\label{leavitt-section}

The goal of this section is to describe an additive monoidal Karoubi closed category $\C$ whose Grothendieck ring is conjecturally isomorphic to $\Z[\frac{1}{2}]$.

\subsection{Category $\C$}
Let $\K$ be a field.
Consider a $\K$-linear pre-additive strict monoidal category $\C$ with one generating object $X$, in addition to the unit object $\mb$.
A set of generating morphisms is given in Figure~\ref{leav1}.

\begin{figure}[h]
    \includegraphics[height=3.7cm]{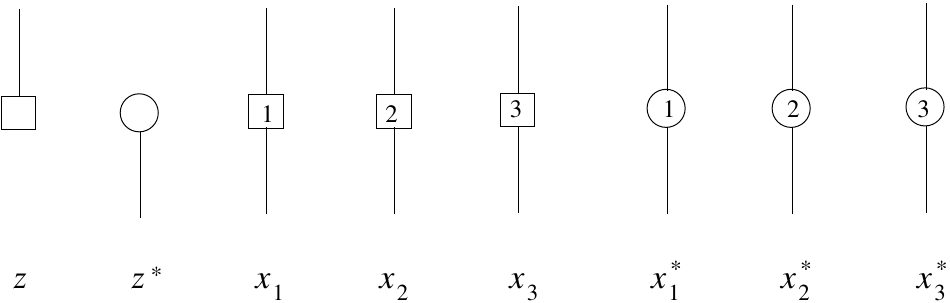}
    \caption{Generating morphisms.}
    \label{leav1}
\end{figure}

Six of these eight generating morphisms are endomorphisms of $X$,
one is a morphism from $\mb$ to $X$, and the last morphism
goes from $X$ to $\mb$.
We denote these generators by $z, \wtz, x_i, x_i^*$ for $i=1,2,3$, from left to right in Figure~\ref{leav1}.
In particular, $x_i, x_i^*$ are endomorphisms of $X$, $z$ a morphism from $\mb$ to $X$, and $\wtz$ a morphism from $X$ to $\mb$.
We draw $x_i$ as a long strand decorated by a box with label $i$, $x_i^*$ as a long strand decorated by a circle with label $i$, and $z$ as a short top strand decorated by an empty box, and $\wtz$ as a short bottom strand decorated by an empty circle, respectively.

A pair of far away generators commute.
Therefore, a horizontal composition of diagrams is independent of their height order.
Given two diagrams $f,g$, let $f\ot g$ denote the horizontal composition of $f$ and $g$, where $f$ is on the left of $g$.

Local relations are given in Figure~\ref{leav2}, where the vertical line is $1_{X}$, and the empty diagram is $1_{\mb}$.
The relations can be written as
\begin{align} \label{rel leavitt}
\begin{split}
& \wtz z= 1_{\mb}, \\
& x_i^* z = 0, \ \ \ \  \wtz x_i=0, \ \ \mbox{for} \ \ i=1,2,3, \\
& x_i^* x_j= \delta_{i,j}1_X, \ \ \mbox{for} \ \ i,j=1,2,3, \\
& \sum\limits_{i=1}^{3}x_i x_i^*+z \wtz = 1_X.
\end{split}
\end{align}

\begin{figure}[h]
    \includegraphics[height=6cm]{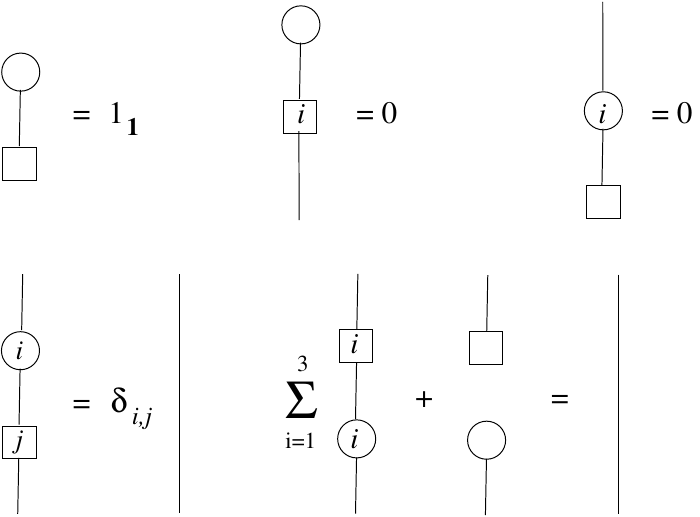}
    \caption{Defining local relations.}
    \label{leav2}
\end{figure}

Let $\Lambda=\bigsqcup\limits_{k\geq0}\Lambda^k=\bigsqcup\limits_{k\geq0}\{1,2,3\}^k$ be the set of sequences of indices, where $\Lambda^0$ consists of a single element of the empty sequence.
For $I=(i_1,\dots,i_k) \in \Lambda^k$, let $I^*=(i_k,\dots,i_1)$, and $|I|=k$.
Let $x_{I}, x_{I}^* \in \End_{\C}(X)$ denote compositions $x_{i_1}\cdots x_{i_k}$ and $x_{i_1}^*\cdots x_{i_k}^*$.
Let $z_{I}=x_{I}z \in \Hom_{\C}(\mb,X)$ and $\wtz_{I}=\wtz x_{I}^* \in \Hom_{\C}(X,\mb)$.
If $|I|=|J|$, then $x_{I}^*x_{J}=\delta_{I,J^*}1_{X}$, and subsequently $z_{I}^*z_{J}=\delta_{I,J^*}1_{\mb}$.

Figure~\ref{leav3} shows our notations for some vertical compositions of generating morphisms.
We draw $x_I$ as a long strand decorated by a box with label $I$, and $x_{I}^*$ as a long strand decorated by a circle with label $I$, respectively.
We draw $z_I$ as a short top strand decorated by a box with label $I$, and $\wtz_I$ as a short bottom strand decorated by a circle with label $I$, respectively.

\begin{figure}[h]
    \includegraphics[height=4cm]{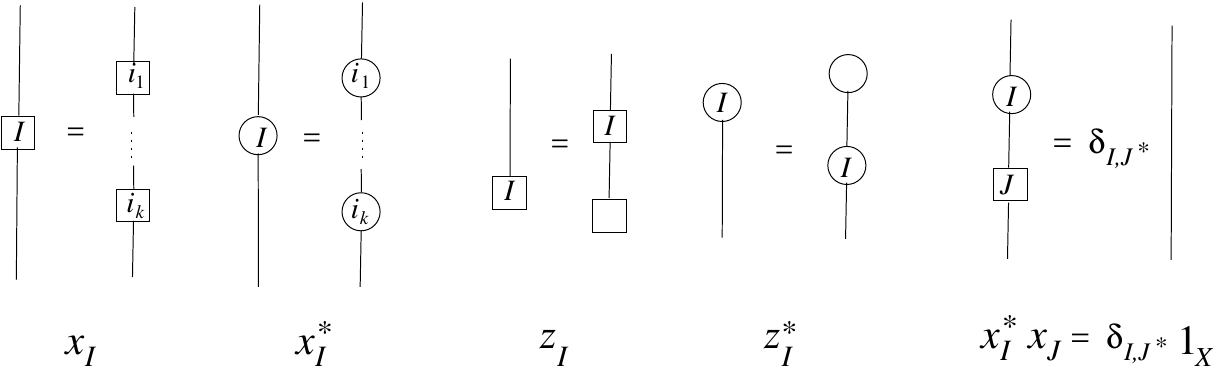}
    \caption{Notations for compositions $x_{I}, x_{I}^*, z_{I}, \wtz_{I}$ for $I=(i_1,\dots,i_k)$, and the relation $x_{I}^*x_{J}=\delta_{I,J^*}1_{X}$ if $|I|=|J|$.}
    \label{leav3}
\end{figure}


\vspace{.2cm}
\n{\bf Bases of morphism spaces.}
We observe that the category $\C$ is generated from the suitable
data $(A,e)$ as described in Section~\ref{basis-section},
where $A$ is a $\K$-algebra.
To see this, we restrict the above diagrams and defining relations on them to the case when there is at most one strand at the top and at most
one strand at the bottom. In other words, we consider generating
morphisms in Figure~\ref{leav1} and compose them only vertically,
not horizontally, with the defining relations in Figure~\ref{leav2}. The idempotent $e$ is given by the
empty diagram, while $1-e$ is the undecorated vertical strand diagram.

It follows from the defining relations (\ref{rel leavitt}) that $A$ is isomorphic to the Leavitt path algebra $L(Q)$ of the graph $Q$ in (\ref{quiver1}).
In particular, as a $\K$-vector space, algebra $A$ has a basis
$$ \{1_{\mb}\} \cup \{ z_{I} ~|~ I \in \Lambda\} \cup \{\wtz_{I} ~|~ I \in \Lambda\} \cup\{z_{I}\wtz_{J} ~|~ I,J \in \Lambda\} \cup \{x_{I}~x_{J}^*~| ~~I, J \in \Lambda, (i_{|I|},j_1)\neq(3,3)\} $$
by~\cite[Corollary 1.5.12]{AAS}.
Our notations for some of these basis elements are shown in
Figure~\ref{leav3}.

The basis of $A$ can be split into the following disjoint subsets:
\be
\item $eAe\cong \K$ has a basis  $\mathbb{B}_{0,0}=\{1_{\mb}\}$ consisting of a single element which is the empty diagram;
\item $(1-e)Ae$ has a basis $\mathbb{B}_{1,0}=\{z_{I} ~|~ I \in \Lambda\}$. Element $z_{I}$
is depicted by a short top strand decorated by a box with label $I$, see Figure~\ref{leav3};
\item $eA(1-e)$ has a basis $\mathbb{B}_{0,1}=\{\wtz_{I} ~|~ I \in \Lambda\}$. Element
$\wtz_{I}$ is depicted by a short bottom strand decorated by a circle with label $I$ (lollipop in Figure~\ref{leav3});
\item $(1-e)A(1-e)$ has a basis $\mathbb{B}_{1,1}(0) \sqcup \mathbb{B}_{1,1}(1)$, where $\mathbb{B}_{1,1}(0)=\{z_{I}\wtz_{J} ~|~ I,J \in \Lambda\}$ consists of pairs (short top strand with a labelled box, short bottom strand with a labelled circle), and $\mathbb{B}_{1,1}(1)=\{x_{I}~x_{J}^*~| ~~I, J \in \Lambda, (i_{|I|},j_1)\neq(3,3)\}$ consists of long strands whose decoration satisfies that no circle is above any box, and no box with label $3$ is next to a circle with label $3$.
\ee
The multiplication map $ (1-e)Ae \otimes eA(1-e)  \ra (1-e)A(1-e)$ sends the basis $\mathbb{B}_{1,0} \times \mathbb{B}_{0,1}$ of
$(1-e)Ae \otimes eA(1-e)$ bijectively to $\mathbb{B}_{1,1}(0)$ so that the multiplication map is injective.

We see that the conditions on $(A,e)$ from
the beginning of Section~\ref{basis-section} are satisfied,
and we can indeed form the monoidal category $\C$ as above with objects
$X^{\otimes n}$, over $n\ge 0$.
Algebra $A$ can then be described as the direct sum
$$A\ \cong \ \End_{\C}(\mb)\oplus \Hom_{\C}(\mb, X) \oplus \Hom_{\C}(X,\mb) \oplus \End_{\C}(X).$$
Therefore, a basis of $\Hom_{\C}(\xn, \xm)$ is given in Theorem \ref{basisthm}.

\vspace{.2cm}
\n{\bf The idempotent completion of $\C$.}
Recall that $\C^{add}$ denotes the additive closure of $\C$, and $\Ka(\C)$ denotes the idempotent completion of $\C^{add}$.
Objects of $\C^{add}$ are finite formal direct sums of nonnegative powers $\xn$, where $X^{\ot 0}=\mb$.
The category $\Ka(\C)$ is $\K$-linear additive strict monoidal.

Defining local relations are chosen to have an isomorphism in $\Ka(\C)$:
\begin{gather} \label{eq iso leavitt}
X \cong \mb \oplus X^{3},
\end{gather}
given by $(\wtz, x_1^*, x_2^*, x_3^*)^T \in \Hom_{\Ka(\C)}(X, \mb \oplus X^{3})$, and $(z,x_1,x_2,x_3) \in \Hom_{\Ka(\C)}(\mb \oplus X^{3},X)$, see Figure~\ref{leav4}.
Tensoring with $X^{\otimes (k-1)}$ in $\Ka(\C)$ on either side of isomorphism (\ref{eq iso leavitt}) results in isomorphisms in $\Ka(\C)$
\begin{gather} \label{eq iso2 leavitt}
X^{\otimes k} \cong X^{\otimes (k-1)} \oplus (X^{\otimes k})^{3}.
\end{gather}

\begin{figure}[h]
    \includegraphics[height=4cm]{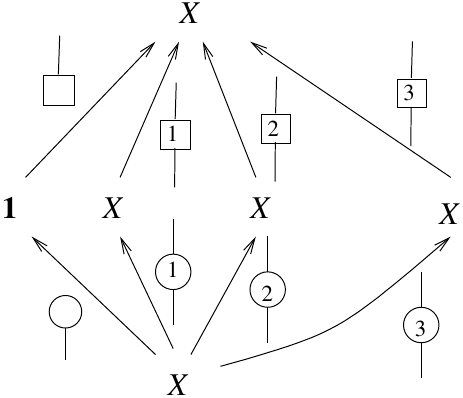}
   \caption{The isomorphism $X \cong \mb \oplus X^{3}$ in $\Ka(\C)$.}
    \label{leav4}
\end{figure}

\vspace{.2cm}
\n {\bf Algebras of endomorphisms.}
Part of the structure of $\Ka(\C)$ can be encoded into an idempotented algebra $B$, which has a complete system of mutually orthogonal idempotents $\{1_n\}_{n\ge 0},$ so that
$$ B \ = \ \bigoplus\limits_{m,n\ge 0} 1_m B 1_n,$$
and
$$ 1_m B 1_n = \Hom_{\Ka(\C)}(X^{\otimes n},X^{\otimes m}).$$
Multiplication in $B$ matches composition of morphisms in $\C$.

We also define
$$ B_k = \bigoplus\limits_{m,n\le k} 1_m B 1_n,$$
which is an algebra with the unit element $\sum\limits_{n \le k} 1_n$.
The inclusion $B_k \subset B_{k+1}$ is nonunital.
The algebra $B_0 \cong \K$.
Define
$$ A_k = 1_k B 1_k = \End_{\Ka(\C)}(\xk),$$
which is an algebra with the unit element $1_k$.
The inclusion $A_k \subset B_k$ is nonunital for $k>0$.

Let $\alpha_k: A_{k-1} \hookrightarrow A_{k}$ be an inclusion of algebras given by tensoring with $z\wtz$ on the left
\begin{gather}
\alpha_k(f)=(z\wtz)\otimes f,
\end{gather}
for $f \in A_{k-1}$.
Note that $\alpha_k$ is nonunital.


\subsection{Towards computing $K_0(\Ka(\C))$}
Let $\C_k$ be the smallest full subcategory of $\C$ which contains the objects $X^{\otimes n}$, $0 \leq n \leq k$.
Let $\Ka(\C_k)$ be the idempotent completion of the additive closure $\C_k^{add}$ of $\C_k$.
There is a family of inclusions of additive categories $\Ka(\C_{k-1})\subset \Ka(\C_k)$.
Similarly, there is a family of inclusions $g_k: \Ka(\C_k) \ra \Ka(\C)$ of additive categories.
We have the analogue of Proposition \ref{prop K0 of C}.

\begin{prop} \label{prop K0 of C leavitt}
There is a natural isomorphism of abelian groups
$$K_0(\Ka(\C)) \cong \varinjlim K_0(\Ka(\C_k)).$$
\end{prop}

For a unital algebra $A$, let $\pr(A)$ denote the additive category of finitely generated projective right $A$-modules.
Let $K_0(A)$ denote the split Grothendieck group of $\pr(A)$.
The category $\C_k$ contains a full subcategory $\C'_k$ with a single object $\xk$ whose endomorphism algebra $\op{End}_{\C_k}(\xk)=A_k$.
Let $\Ka(\C'_k)$ be the idempotent completion of the additive closure ${\C'_k}^{add}$ of $\C'_k$.
Thus, the category $\Ka(\C'_k)$ is isomorphic to $\pr(A_k)$.
There is an inclusion $h_k: \pr(A_k) \subset \Ka(\C_k)$ of additive categories.
Isomorphism (\ref{eq iso2 leavitt}) implies that $h_k$ is an equivalence.
Therefore, ${h_k}_*: K_0(A_k) \ra K_0(\Ka(\C_k))$ is an isomorphism.
By Proposition \ref{prop K0 of C leavitt}, there is a natural isomorphism of abelian groups:
\begin{gather} \label{iso K0 of C leavitt}
K_0(\Ka(\C)) \cong \varinjlim K_0(A_k).
\end{gather}
Here $A_k$ is just a $\K$-algebra without the grading and the differential, and $K_0(A_k)$ is the usual Grothendieck group of the ring.
The major part of the complexity in this construction lies in dealing
with algebras of exponential growth, including the endomorphism algebra
$A_k$ of the object $X^{\otimes k}$.

\vspace{.2cm}
\n{\bf An approach to $K_0(A_k)$.}
Recall the chain of two-sided ideals $J_{n,k}$ from (\ref{def ideal}), and the quotient algebra $L_k=A_k / J_{k-1,k}$ from (\ref{def lk}) in Section \ref{basis-section}.
The algebra $L_k$ is naturally isomorphic to $L^{\otimes k}$ for $L=L_1$.

For $k=1$, $A_1=J_{1,1}$ has a basis $\BB_{1,1}=\BB_{1,1}(0) \sqcup \BB_{1,1}(1)$, and $J=J_{0,1}$ has a basis $\BB_{1,1}(0)$ with respect to the inclusion $J_{0,1} \subset A_1$, by Corollary \ref{cor-basis-ideal}.
Under the quotient map $A_1 \ra L$, the set $\BB_{1,1}(1)$ is mapped bijectively to the \emph{normal form} basis of $L=L_1$, see \cite[Section 5]{Co2}.
The algebra $L$ is naturally isomorphic to the Leavitt algebra $L(1,3)$.
Thus,
$$L_k \cong L(1,3)^{\otimes k}.$$

If we view $A_k$ as a DG algebra concentrated in degree $0$ with the trivial differential, the analogue of Lemma \ref{lem exact} still holds.
Therefore, there is an induced exact sequence of K-groups
\begin{gather} \label{les leavitt}
K_1(\lk) \xra{\bdry} K_0(A_{k-1}) \xra{{i_k}_*} K_0(A_k) \xra{{j_k}_*} K_0(\lk).
\end{gather}

\begin{conj} \label{conj lk}
For $k \geq 1$, $K_0(\lk)$ is isomorphic to $\Z/2$ with a generator $[\lk]$, and $K_1(\lk)$ is torsion.
\end{conj}

The Leavitt algebra $L_1$ is regular supercoherent by \cite[Lemma 6.1]{ArCo}.
The conjecture is known to be true for $k=1,2$, see \cite[Theorem 7.6]{ArBrCo}.

By an argument similar to that in the proof of Theorem \ref{thm main}, if Conjecture~\ref{conj lk} is true, then there is a ring isomorphism
$$K_0(\Ka(\C)) \cong \Z\bigg[\frac{1}{2} \bigg].$$

\vspace{.2cm}
\n{\bf Categorical actions of $\Ka(\C)$.}
There is an action $F_m: \Z[\frac{1}{2}] \times \Z/(2m+1) \ra \Z/(2m+1)$ of the ring $\Z[\frac{1}{2}]$ on the abelian group $\Z/(2m+1)$, where $-\frac{1}{2}$ acts as multiplication by $m$.

Recall that $L(m,n)$ is the $\K$ algebra generated by entries of $x_{ij}, y_{ij}$ of matrices $X=(x_{ij}), Y=(y_{ij})$ of size $m\times n$ and $n\times m$ respectively, subject to the relation: $XY=I_m, YX=I_n$.
The algebra $L(m,n)$ is the universal object with respect to the non-IBN (Invariant Basis Number) property: $R^m \cong R^n$.
It is known \cite{AG} that $K_0(L(m,n))\cong \Z/(n-m)$ generated by the class $[L(m,n)]$.

There is a family of categorical actions
$$\cal{F}_m: \Ka(\C) \times \cal{P}(L(m,3m+1)) \ra \cal{P}(L(m,3m+1))$$
of $\Ka(\C)$ on the category of finitely generated projective right $L(m,3m+1)$-modules,
where the generating object $X$ of $\Ka(\C)$ acts by tensoring with the $L(m,3m+1)$ bimodule $L(m,3m+1)^{m}$.
Conjecturally, $\cal{F}_m$ categorifies the linear action $F_m$.


\subsection{A possible categorification of $\Zen$} \label{Sec Zen}
We consider a pre-additive $\K$-linear strict monoidal category $\C$ with a generating object $X$ in addition to the unit object $\mb$ and require the following isomorphism
\begin{equation}
X\cong  \mb \oplus X^{n+1}.
\end{equation}
This isomorphism is a natural generalization of isomorphism (\ref{eq iso leavitt}). Generating morphisms that induce these mutually-inverse isomorphisms  are denoted $z \in \Hom_{\C}(\mb, X), \wtz \in \Hom_{\C}(X, \mb)$, and $x_i,x_i^* \in \End_{\C}(X)$ for $1\leq i \leq n+1$.
The defining relations, generalizing relations (\ref{rel leavitt}), are
\begin{align} \label{rel leavitt n}
\begin{split}
& \wtz z= 1_{\mb}, \\
& x_i^* z = 0, \ \ \ \  \wtz x_i=0, \ \ \mbox{for} \ \ 1\leq i \leq n+1, \\
& x_i^* x_j= \delta_{i,j}1_X, \ \ \mbox{for} \ \ 1\leq i,j \leq n+1, \\
& \sum\limits_{i=1}^{n+1}x_i x_i^*+z \wtz = 1_X.
\end{split}
\end{align}

Let $A$ denote the algebra generated by $1_{\mb}, 1_X, z, \wtz, x_i, x_i^*,$ $1\le i \le n+1$, subject to the relations above and
the obvious compatibility relations between generators
$z,z^{\ast},x_i,x_i^*$ and idempotents $1_{\mb}, 1_X$,
for instance, $z 1_{\mb} = z = 1_X z$ and $z 1_X = 0 = 1_{\mb} z$.
.
The data $(A, e=1_{\mb})$ satisfies the conditions described at the beginning of Section~\ref{basis-section}.
Thus, we can form a pre-additive monoidal category $\C$ and recover
bases of morphisms between tensor powers of $X$ from suitable bases
of $A$ compatible with the idempotent decomposition $1=e+(1-e)$,
as explained in Section~\ref{basis-section}.
Let $\Ka(\C)$ denote the idempotent completion of the additive closure of $\C$. Category $\Ka(\C)$ is an additive $\K$-linear Karoubi closed monoidal category.
\begin{conj} \label{conj n}
There is a ring isomorphism $K_0(\Ka(\C)) \cong \Zen$.
\end{conj}


\end{document}